%% file: final.tex
\documentclass[reqno,11pt]{amsart}

\usepackage{amsmath,amssymb,amsthm,mathtools}
\usepackage{mathrsfs}
\usepackage{bbm}

\usepackage{graphicx}
\usepackage{tikz}

\usetikzlibrary{backgrounds}
\usetikzlibrary{arrows}
\usetikzlibrary{shapes.geometric}

\pgfdeclarelayer{background}
\pgfdeclarelayer{edgelayer}
\pgfdeclarelayer{nodelayer}
\pgfsetlayers{background,edgelayer,nodelayer,main}

\tikzstyle{Labeled}=[
    fill=red!40,
    draw=black,
    circle
]

\tikzstyle{Unlabeled}=[
    fill=blue!25,
    draw=black,
    rectangle
]

\tikzstyle{Neighbor}=[
    -,
    draw=black
]

\usepackage[usenames,dvipsnames]{pstricks}
\usepackage{pst-grad}
\usepackage{pst-plot}
\usepackage[space]{grffile} 

\usepackage{verbatim}
\usepackage{xcolor}
\usepackage{cancel}
\usepackage[normalem]{ulem}
\usepackage{enumitem}
\usepackage{etoolbox}

\usepackage[colorlinks=true,allcolors=blue]{hyperref}

\numberwithin{equation}{section}

\newtheorem{theorem}{Theorem}[section]
\newtheorem{lemma}[theorem]{Lemma}
\newtheorem{corollary}[theorem]{Corollary}
\newtheorem{proposition}[theorem]{Proposition}

\theoremstyle{definition}
\newtheorem{definition}[theorem]{Definition}
\newtheorem{example}[theorem]{Example}

\theoremstyle{remark}
\newtheorem{remark}[theorem]{Remark}

\title{The Monge--Amp\`ere equation on graphs}
 
\author[A. Alkhozaae]{Ahmed Alkhozaae}
\address{Applied Mathematics and Computational Sciences (AMCS), Computer, Electrical and Mathematical Sciences and Engineering Division (CEMSE), King Abdullah University of Science and Technology (KAUST), Thuwal, 23955-6900, Kingdom of Saudi Arabia}
\email{ahmed.abdali.1@kaust.edu.sa}

\author[J.D. Rossi]{Julio D. Rossi}
\address{Departamento de Matemática y Estadística, Universidad Torcuato Di Te\-lla,  Av. Figueroa Alcorta 7350 (C1428BCW), Ciudad de Buenos Aires, Argentina}
\email{julio.rossi@utdt.edu}

\author[A. Sobral]{Aelson Sobral}
\address{Applied Mathematics and Computational Sciences (AMCS), Computer, Electrical and Mathematical Sciences and Engineering Division (CEMSE), King Abdullah University of Science and Technology (KAUST), Thuwal, 23955-6900, Kingdom of Saudi Arabia}
\email{aelson.sobral@kaust.edu.sa}

\author[J.M. Urbano]{Jos\'{e} Miguel Urbano\textsuperscript{\dag}}
\thanks{\textsuperscript{\dag}Corresponding author.}
\address{Applied Mathematics and Computational Sciences (AMCS), Computer, Electrical and Mathematical Sciences and Engineering Division (CEMSE), King Abdullah University of Science and Technology (KAUST), Thuwal, 23955-6900, Kingdom of Saudi Arabia; and CMUC, Department of Mathematics, University of Coimbra, 3000-143 Coimbra, Portugal}
\email{miguel.urbano@kaust.edu.sa}

\begin{document}




  
\begin{abstract}
We introduce a version of the Monge--Amp\`ere equation on finite graphs, motivated by nonlinear graph-based interpolation and semi-supervised learning. The operator is defined as the product of discrete analogs of the Hessian eigenvalues, obtained via local order statistics of function values at neighboring vertices. We derive an equivalent Bellman-type formulation of the inhomogeneous Dirichlet problem, establish a comparison principle and uniqueness in the strictly graph-convex class, and investigate existence via Perron's method, identifying certain graph-theoretic obstructions. We also study the homogeneous equation, for which the problem reduces to a nonlinear interpolation rule involving the smallest discrete eigenvalue. Finally, we propose numerical schemes for both the homogeneous and inhomogeneous problems.
\end{abstract}
 
\date{\today}

\maketitle
 
\tableofcontents 

\section{Introduction}\label{sct:intro}

Graph-based semi-supervised learning concerns extending labels prescri\-bed on a small subset of a data set to a much larger collection of unlabeled vertices, using the geometry encoded by all available data. More precisely, given a finite point cloud \(\mathcal X\), a labeled subset \(\mathcal O \subset \mathcal X\), and boundary datum \(g \colon \mathcal O \to \mathbb R\), one first equips \(\mathcal X\) with a graph structure reflecting the affinities between data points. The goal is then to construct an extension \(u \colon \mathcal X \to \mathbb R\), satisfying \(u=g\) on \(\mathcal O\), whose values on \(\mathcal X \setminus \mathcal O\) are compatible with the geometry of the graph.

A natural and by now classical way to impose this compatibility is through a harmonic extension, or equivalently, through graph-Laplacian regularization. In this approach, the unknown label function is selected by minimizing a quadratic graph energy subject to the boundary condition \(u=g\) on \(\mathcal O\), leading to a linear Dirichlet problem on the graph. This viewpoint appears, for instance, in the Gaussian-field formulation of Zhu, Ghahramani, and Lafferty \cite{ZhuGhahramaniLafferty2003}, in the local and global consistency method of Zhou et al. \cite{ZhouBousquetLalWestonScholkopf2004}, and in the manifold-regularization framework of Belkin, Niyogi, and Sindhwani \cite{BelkinNiyogiSindhwani2006}. These methods are linear, computationally efficient, and have been highly influential in machine learning. However, their reliance on the graph Laplacian makes them intrinsically diffusive: local information is averaged over all graph directions, which can oversmooth sharp transitions and lead to degeneracies in very low-label regimes. These limitations have motivated the development of modified and nonlinear graph PDEs for learning, including low-label-rate analysis for graph-Laplacian regularization \cite{NadlerSrebroZhou2009, CalderSlepcevThorpe2023}, Poisson learning \cite{CalderCookThorpeSlepcev2020}, \(p\)-Laplacian and Lipschitz-learning methods \cite{ElAlaouiChengRamdasWainwrightJordan2016, SlepcevThorpe2019, Calder2019Game, Calder2019Lipschitz, FloresCalderLerman2022}, and Hamilton--Jacobi equations on graphs \cite{CalderEttehad2022}.

In this paper, we pursue a different nonlinear direction, replacing the diffusive graph-Laplacian paradigm by a Monge--Amp\`ere-type equation on finite graphs. In the continuum setting, the Monge--Amp\`ere operator is given by
\[
    u \longmapsto \det(D^2 u),
\]
that is, by the product of the eigenvalues of the Hessian. In contrast with the Laplacian, which is the trace of the Hessian, the Monge--Amp\`ere operator is fully nonlinear and is degenerate elliptic only inside the cone of convex functions. It plays a central role in convex geometry, fully nonlinear elliptic PDE, and optimal transport; see, for example, \cite{Gutierrez2016}. It has also motivated the development of monotone numerical schemes for fully nonlinear equations and optimal transportation problems; cf. \cite{Oberman2008, FroeseOberman2011, BenamouFroeseOberman2014}. From the perspective of data analysis, optimal transport and Monge--Amp\`ere-type equations offer geometry-sensitive mechanisms that are fundamentally different from isotropic diffusive smoothing; see \cite{PeyreCuturi2019}. This makes the Monge--Amp\`ere operator a natural candidate for graph-based interpolation and semi-supervised learning, particularly in settings where label propagation should respond to anisotropic structure in the data rather than simply average over local graph directions.

The central difficulty in formulating a Monge--Amp\`ere equation on a finite graph is that there is no canonical notion of Hessian. We address this by replacing the eigenvalues of the Hessian with discrete eigenvalues defined through local order statistics of the values of \(u\) on each neighborhood. Let \(x\in \mathcal X\), and suppose that its neighborhood \(\mathcal N_x\) has even cardinality \(n\). For \(i=1,\dots,n/2\), we define
\[
    \lambda_i[u](x)
    \coloneqq
    \min_{\substack{S\subseteq \mathcal N_x\\ |S|=2i}}
    \left(
        \max_{\substack{y,z\in S\\ y\neq z}}
        \frac{u(y)+u(z)}{2}
        -u(x)
    \right).
\]
This definition has a simple order-statistic representation. Indeed, if the neighbors of \(x\) are indexed so that
\[
    u(y_1)\leq u(y_2)\leq \cdots \leq u(y_n),
\]
then
\[
    \lambda_i[u](x)
    =
    \frac{u(y_{2i-1})+u(y_{2i})}{2}-u(x),
    \qquad i=1,\dots,\frac n2 .
\]
Thus, the quantities \(\lambda_i[u](x)\) play the role of ordered directional second-order increments at the vertex \(x\). They provide a discrete analog of the Hessian eigenvalues, with the smallest increments detecting the most degenerate local directions. In fact, recall the well-known formula for the eigenvalues of a symmetric matrix
$$
\lambda_i (D^2 u) = 
\min_{\substack{S\subseteq \mathbb{R}^d \\ \dim S=i}}
    \left(
        \max_{\substack{v\in S \\ \| v \| =1}}
        \langle D^2 u \, v , v \rangle
    \right)
    = 
\min_{\substack{S\subseteq \mathbb{R}^d \\ \dim S=i}}
    \left(
        \max_{\substack{v\in S \\ \| v \| =1}}
        \partial_{vv} u
    \right).
$$

In addition to this similarity, these discrete eigenvalues are naturally compatible with the graph Laplacian. With the sign convention
\[
    \mathcal L u(x)
    =
    \frac{1}{n}\sum_{y\in \mathcal N_x}u(y)-u(x),
\]
one has
\[
    \mathcal L u(x)
    =
    \frac{2}{n}\sum_{i=1}^{n/2}\lambda_i[u](x).
\]
In this sense, the graph Laplacian is the trace-type operator associated with the discrete eigenvalues \(\lambda_i[u]\). The graph Monge--Amp\`ere operator introduced in this paper is instead the corresponding determinant-type operator,
\[
    \mathcal M [u](x)
    \coloneqq
    \prod_{i=1}^{n/2}\lambda_i[u](x).
\]

This construction also suggests a natural notion of convexity on the graph (we refer to \cite{convexT,convexTT, DelPezzoMosqueraRossi2014} for related references). We say that \(u\) is graph convex if, for every vertex \(x\) and every pair of distinct neighbors \(y,z\in \mathcal N_x\),
\[
    \frac{u(y)+u(z)}{2}\geq u(x).
\]
Equivalently, the associated discrete eigenvalues satisfy
\[
    \lambda_i[u](x)\geq 0,
    \qquad i=1,\dots,\frac{|\mathcal N_x|}{2},
\]
at every vertex \(x\) where the operator is defined. This positivity condition is the discrete analog of the convexity constraint in the continuum Monge--Amp\`ere equation, under which the operator is degenerate elliptic. It is therefore not merely a formal requirement, but the condition that places the graph Monge--Amp\`ere operator in its elliptic regime. Moreover, as \(\mathcal M\) is defined as a product of the eigenvalue-like quantities \(\lambda_i[u]\), it is sensitive to degeneracy in any single local direction: one small eigenvalue can force the determinant to be small. This is fundamentally different from graph-Laplacian regularization, where neighboring information is combined through an arithmetic average and such directional degeneracies can be obscured.

We then study the Dirichlet problem
\[
\begin{cases}
    \mathcal M [u](x)=f(x), & x\in \mathcal X\setminus \mathcal O,\\
    u(x)=g(x), & x\in \mathcal O,
\end{cases}
\]
where \(f>0\) on \(\mathcal X\setminus \mathcal O\). A key observation is that, for the graph-convex set of functions, the graph Monge--Amp\`ere equation admits a Bellman-type reformulation. Indeed, for \(a_i\geq 0\), the arithmetic--geometric mean identity gives
\[
    \left(\prod_{i=1}^k a_i\right)^{1/k}
    =
    \inf_{\substack{\alpha_i>0 \\ \prod_{i=1}^k\alpha_i=1}}
    \frac1k\sum_{i=1}^k\alpha_i a_i .
\]
Applying this identity to the discrete eigenvalues \(\lambda_i[u](x)\) rewrites the deter\-minant-type equation as an infimum of weighted trace-type expressions. More precisely, define
\[
    H_i[u](x)
    \coloneqq
    \min_{\substack{S\subseteq \mathcal N_x\\ |S|=2i}}\;
    \max_{\substack{y,z\in S\\ y\neq z}}
    \frac{u(y)+u(z)}{2},
    \qquad i=1,\dots,\frac n2,
\]
where \(n=|\mathcal N_x|\). Since
\[
    \lambda_i[u](x)=H_i[u](x)-u(x),
\]
the equation $$\mathcal M [u](x)=f(x) $$ is equivalent to
\[
    u(x)
    =
    \inf_{\alpha\in\mathcal A_{n/2}}
    \left[
        \frac{
            \sum_{i=1}^{n/2}\alpha_i H_i[u](x)
            -\frac n2 f(x)^{2/n}
        }{
            \sum_{i=1}^{n/2}\alpha_i
        }
    \right],
\]
where
\[
    \mathcal A_{n/2}
    =
    \left\{
        \alpha\in(0,\infty)^{n/2}:
        \prod_{i=1}^{n/2}\alpha_i=1
    \right\}.
\]
This formulation is particularly useful because it makes the operator's monotonicity transparent. For each fixed choice of weights, the right-hand side is built from monotone order-statistic operators \(H_i\). It therefore places the graph Monge--Amp\`ere equation in a framework analogous to Bellman equations for fully nonlinear elliptic PDE, where monotonicity is central both to the viscosity-solution theory and to the convergence of approximation schemes \cite{CrandallIshiiLions1992, BarlesSouganidis1991}.

The first part of the paper develops the basic theory of the inhomogeneous graph Monge--Amp\`ere equation. We prove the order-statistic formula for the discrete eigenvalues, establish their relation with the graph Laplacian, derive the Bellman formulation, and prove a comparison principle for subsolutions and supersolutions. As a consequence, strictly graph-convex solutions are unique whenever they exist. We then formulate a Perron method: under suitable boundedness and nonemptiness assumptions on the subsolution class, the supremum of all subsolutions is the unique solution. At the same time, the inhomogeneous problem exhibits genuine graph-theoretic obstructions. In particular, the existence of barriers for the lower extremal operator is tied to the presence or absence of closed substructures in the unlabeled graph. This shows that, unlike the classical finite-dimensional linear Dirichlet problem for the graph Laplacian, the graph Monge--Amp\`ere equation is sensitive to the combinatorial geometry of the graph.

The second part of the paper studies the homogeneous problem
\[
\begin{cases}
    \mathcal M [u](x)=0, & x\in \mathcal X\setminus \mathcal O,\\
    u(x)=g(x), & x\in \mathcal O.
\end{cases}
\]
For convex functions, the equation
$\mathcal M [u](x)=0$ reduces to the smallest discrete eigenvalue,
\[
    \lambda_1[u]=0,
\]
or equivalently
\[
    u = H_1[u], \quad \text{in} \quad \mathcal X\setminus\mathcal O .
\]
This equation has a simpler structure than the inhomogeneous problem while still retaining the nonlinear, order-statistic character of the Monge--Amp\`ere operator. We prove a comparison principle and uniqueness for the corresponding Dirichlet problem and discuss existence through a Perron construction. The homogeneous equation may be viewed as a nonlinear interpolation rule based on extremal local averages, in contrast with harmonic extension, which is based on full local averaging.

Finally, we discuss numerical schemes motivated by the Bellman formulation. In the inhomogeneous case, the minimization over $\alpha$ can be reduced to a scalar nonlinear equation for the optimal value $t^*$, which is then used to define an iterative update. In the homogeneous case, the update is simpler and is driven by the residual
\[
    u-H_1[u].
\]
These schemes are intended for graph-based interpolation and semi-supervi\-sed learning on point clouds, including sparse-label regimes where purely Laplacian methods may be inadequate. Related sparse-data interpolation problems have motivated weighted nonlocal Laplacian methods \cite{ShiOsherZhu2017}, while nonlinear graph PDEs such as $p$-Laplacian, infinity-Laplacian, Poisson, and Hamilton--Jacobi models provide important points of comparison \cite{CalderCookThorpeSlepcev2020, SlepcevThorpe2019, Calder2019Game, Calder2019Lipschitz, CalderEttehad2022}. The contribution of this paper is to add a determinant-type graph operator to this nonlinear PDE toolkit and to initiate the analysis of its comparison, uniqueness, existence, and computational properties.

\section{Preliminaries}\label{sct:prelim}
Let \(\mathcal G=(\mathcal X,E)\) be a finite simple undirected graph,
and let \(\mathcal O\subset\mathcal X\) be the set of boundary
vertices. For each \(x\in\mathcal X\), let \(\mathcal N_x\) denote
the set of neighbors of \(x\). We assume that there exists a fixed
even integer \(n\ge2\) such that
\[
    |\mathcal N_x|=n,
    \qquad x\in\mathcal X\setminus\mathcal O.
\]
No degree condition is imposed on the boundary vertices.

Throughout this paper, we will consider the following notion of eigenvalues.

\begin{definition}
Let $u \colon \mathcal{X}\rightarrow\mathbb{R}$, and \(x\in\mathcal X\setminus\mathcal O\), with $n$ neighbors. For each $i \in \{1,\dots, n/2\}$, we define
\[
    \lambda_i[u](x)\coloneqq \min_{\substack{S\subseteq \mathcal{N}_x\\ |S|=2i}} \left( \max_{\substack{y,z\in S\\ y\ne z}} \frac{u(y)+u(z)}{2} - u(x) \right).
\]
\end{definition}

The next result gives a simple expression for each eigenvalue whenever we can order the values of $u$ at the neighbors.

\begin{lemma}\label{lemma:eigen-expression}
Let $u \colon \mathcal{X}\rightarrow\mathbb{R}$, and \(x\in\mathcal X\setminus\mathcal O\), with neighbors $\{y_1,\dots,y_n\}$. Assuming that $u(y_1) \leq \dots \leq u(y_n)$, we have
\[
    \lambda_i[u](x) = \frac{1}{2}\bigl(u(y_{2i-1})+u(y_{2i})\bigr)-u(x),\qquad i \in \{1,\dots,n/2\}.
\]
\end{lemma}

\begin{proof}
Let $S\subseteq \mathcal{N}_x$ such that $|S|=2i$. We can write $S = \{y_{k_j}\}_{j=1}^{2i}$, where the indices $(k_j)_{j=1}^{2i}\in\mathbb{N}^{2i}$ are ordered increasingly. So $u(y_{k_j})$ will be increasing in $(k_j)_{j=1}^{2i}$ and then
\[
\max_{\substack{y,z\in S\\ y\ne z}} \frac{u(y)+u(z)}{2} = \frac{u(y_{k_{2i-1}})+u(y_{k_{2i}})}{2}.
\]
Now, we can deduce that
\[
\min_{\substack{S\subseteq \mathcal{N}_x\\ |S|=2i}} \left(\frac{u(y_{k_{2i-1}})+u(y_{k_{2i}})}{2} \right) =\frac{u(y_{2i-1})+u(y_{2i})}{2},
\]
as the minimizing set must be $\{y_1,y_2,\dots,y_{2i-1},y_{2i}\}$, because the last two indices are the smallest feasible indices; any other set would require adding a $y_k$ with a bigger index.

Now, we can take $u(x)$ outside the min-max to get the result:
\[
\begin{aligned}
    \lambda_i[u](x)= \min_{\substack{S\subseteq \mathcal{N}_x\\ |S|=2i}} \left( \max_{\substack{y,z\in S\\ y\ne z}} \frac{u(y)+u(z)}{2} \right) - u(x)=\frac{u(y_{2i-1})+u(y_{2i})}{2}- u(x).
\end{aligned}
\]
\end{proof}

Next, we show that these discrete eigenvalues behave as expected and sum up to the graph Laplacian.

\begin{proposition}
Let $u \colon \mathcal{X}\rightarrow\mathbb{R}$, and \(x\in\mathcal X\setminus\mathcal O\). Let $\mathcal{L}$ be the normalized graph Laplacian on $\mathcal{G}$. Then   
\[
    \mathcal{L}[u](x) = \frac{2}{n} \sum_{i=1}^{n/2}\lambda_i[u](x).
\]
\end{proposition}

\begin{proof}
The normalized graph Laplacian of \(u\) at \(x\) is
$$\mathcal{L}[u](x)=\frac{1}{n} \sum_{i=1}^n u(y_i)-u(x),$$
where $\{y_1,\dots,y_n\}$ are the $n$ neighbors of $x$. Grouping consecutive terms gives
$$\begin{aligned}
    \mathcal{L}[u](x)=&\;\frac{1}{n}u(y_1)+\frac{1}{n}u(y_2)+\dots +\frac{1}{n}u(y_{n-1})+\frac{1}{n}u(y_n)-u(x)\\
    =&\; \frac{2}{n}\left(\frac{u(y_1)+u(y_2)}{2}\right)+\dots +\frac{2}{n}\left(\frac{u(y_{n-1})+u(y_n)}{2}\right)-\frac{n}{2}\frac{2}{n}u(x)\\
    =&\; \frac{2}{n}\left(\frac{u(y_1)+u(y_2)}{2}-u(x)\right)+\dots +\frac{2}{n}\left(\frac{u(y_{n-1})+u(y_n)}{2}-u(x)\right)\\
    =&\;\frac{2}{n}\left(\lambda_1[u](x)+\dots+\lambda_{n/2}[u](x)\right).
\end{aligned}$$
\end{proof}

Finally, we will be using the following notion of convexity on graphs.

\begin{definition}[Graph convexity]
Let $A\subseteq\mathcal X$ and let $u\colon\mathcal X\to\mathbb R$. We say that $u$ is graph convex on $A$ if
\[
    \frac{u(y)+u(z)}{2}\ge u(x)
\]
for every $x\in A$ and every pair of distinct vertices $y,z\in\mathcal N_x$. We say that $u$ is strictly graph convex on $A$ if the inequality is strict.
\end{definition}

\section{A Monge--Amp\`ere equation on a graph}

In this section, we study a Monge--Amp\`ere-type equation on the graph $\mathcal G$.  Given a function $u\colon \mathcal X\to \mathbb R$, we define the graph Monge--Amp\`ere operator by
\[
    \mathcal M [u](x)
    \coloneqq
\prod_{i=1}^{n/2}\lambda_i[u](x),
\]
where $n=|\mathcal N_x|$ is assumed to be even.  

Let $\mathcal O\subset \mathcal X$ be the set of boundary vertices, let $g\colon \mathcal O\to \mathbb R$ be a prescribed boundary datum, and let $f\colon \mathcal X\setminus \mathcal O\to (0,\infty)$.  We consider the Dirichlet problem
\[
\begin{cases}
    \mathcal M [u](x)=f(x), & x\in\mathcal X\setminus\mathcal O,\\
    u(x)=g(x), & x\in\mathcal O,
\end{cases}
\]
and seek a solution that is strictly graph convex on $\mathcal X\setminus\mathcal O$. The strict convexity assumption ensures that the discrete eigenvalues are positive, which is the natural ellipticity regime for the Monge--Amp\`ere operator.

The product form of the equation is not the most convenient for comparison arguments or for constructing solutions.  We therefore rewrite it in an equivalent Bellman-type form.  For $k\in\mathbb N$, set
\[
    \mathcal A_k
    \coloneqq
    \left\{
        \alpha\in(0,\infty)^k:
        \prod_{i=1}^k \alpha_i=1
    \right\}.
\]
We use the elementary identity
\begin{equation}\label{eq:infmean}
    k\left(\prod_{i=1}^k a_i\right)^{1/k}
    =
    \inf_{\alpha\in\mathcal A_k}
    \sum_{i=1}^k \alpha_i a_i,
    \qquad a_i\geq 0,
\end{equation}
which is an optimized form of the arithmetic--geometric mean inequality. Applying \eqref{eq:infmean} to the discrete eigenvalues gives
\[
    \frac{n}{2}\bigl(\mathcal M [u](x)\bigr)^{2/n}
    = \frac{n}{2}\Big(
\prod_{i=1}^{n/2}\lambda_i[u](x) \Big)^{2/n}
=    \inf_{\alpha\in\mathcal A_{n/2}}
    \sum_{i=1}^{n/2}\alpha_i \lambda_i[u](x).
\]
It remains to express the eigenvalues in a form that separates the dependence on $u(x)$.  We define
\[
    H_i[u](x)
    \coloneqq
    \min_{\substack{S\subseteq \mathcal N_x\\ |S|=2i}}
    \max_{\substack{y,z\in S\\ y\neq z}}
    \left(
        \frac{u(y)+u(z)}{2}
    \right)
\]
and then we have
\[
    \lambda_i[u](x) = H_i[u](x) - u(x).
\]
Therefore, the Dirichlet problem can be written as
\begin{equation}\label{eq:M-A}
    \begin{cases}
        \displaystyle
        \inf_{\alpha\in\mathcal A_{n/2}}
        \sum_{i=1}^{n/2}
        \alpha_i\bigl(H_i[u](x)-u(x)\bigr)
        =
        \frac{n}{2} (f(x))^{2/n},
        & x\in\mathcal X\setminus\mathcal O,\\[2ex]
        u(x)=g(x),
        & x\in\mathcal O.
    \end{cases}
\end{equation}

Although \eqref{eq:M-A} is already closer to a fully nonlinear elliptic formulation, it is still not the most useful representation. In the next subsection, we isolate $u(x)$ from the infimum in order to obtain a genuine Bellman-type equation. This formulation will be the main tool for the comparison principle and the Perron construction. 

For the remainder of the paper, for \(z\colon\mathcal{X}\to\mathbb{R}\), \(x\in \mathcal{X} \setminus \mathcal{O}\), and \(\alpha\in\mathcal{A}_{n/2}\), define
\begin{equation}\label{eq:definition-Qf}
    \mathcal{Q}_z(x;\alpha)
    \coloneqq
    \frac{
        \displaystyle\sum_{i=1}^{n/2}\alpha_i H_i[z](x)
        -\frac n2 f(x)^{2/n}
    }{
        \displaystyle\sum_{i=1}^{n/2}\alpha_i
    }.
\end{equation}
The associated Bellman operator is
\begin{equation}\label{eq:definition-Ff}
    F[z](x)
    \coloneqq
    \inf_{\alpha\in\mathcal{A}_{n/2}}
    \mathcal{Q}_z(x;\alpha),
    \qquad x\in\mathcal{X} \setminus \mathcal{O}.
\end{equation}

\subsection{Bellman formulation}

We begin with the following proposition that ensures that, for strictly graph-convex functions, we can restrict the infimum in the arithmetic–geometric mean representation to a smaller set. This mirrors the continuous Monge--Amp\`ere setting, where strict convexity permits the Bellman minimization to be restricted to a compact uniformly elliptic class of matrices.

\begin{proposition}\label{lemma:restrict-set-alpha}
Let \(u\colon\mathcal X\to\mathbb R\) be strictly graph-convex on
\(\mathcal X\setminus\mathcal O\), fix
\(x\in\mathcal X\setminus\mathcal O\), and choose \(\theta>0\) such that
\[
    \theta > \frac{\mathcal{L}[u](x)}{\frac{2}{n}\lambda_1[u](x)}.
\]
If we define $\mathcal{A}_{n/2}(\theta) \coloneqq \left\{\alpha \in \mathcal{A}_{n/2}\colon \alpha_i \leq \theta, \ i=1,\ldots,n/2\right\}$, then we have  
\[
    \inf_{\alpha\in\mathcal{A}_{n/2}}\left(\frac{2}{n}\sum_{i=1}^{n/2}\alpha_i \lambda_i[u](x) \right) = \inf_{\substack{\alpha\in\mathcal{A}_{n/2}(\theta)}}\left(\frac{2}{n}\sum_{i=1}^{n/2}\alpha_i \lambda_i[u](x)\right).
\]
In particular, the infimum is attained; there is $\alpha^* \in \mathcal{A}_{n/2}(\theta)$ such that 
\[
    \mathcal{M}[u](x) = \left(\frac{2}{n}\sum_{i=1}^{n/2}\alpha^*_i \lambda_i[u](x) \right)^{n/2}
\]
\end{proposition}

\begin{proof}
Define
\[
    \mathcal{S}(\alpha) \coloneqq \frac{2}{n}\sum_{i=1}^{n/2}\alpha_i \lambda_i[u](x).
\]
Since $\mathcal{A}_{n/2}(\theta) \subset \mathcal{A}_{n/2}$, we only need to prove that
\begin{equation}\label{ineq:inf-realized-smaller-set}
    \inf_{\alpha\in\mathcal{A}_{n/2}} \mathcal{S}(\alpha) \geq \inf_{\substack{\alpha\in\mathcal{A}_{n/2}(\theta)}} \mathcal{S}(\alpha).
\end{equation}
We begin by observing that  
\[
   \mathcal{S}(\alpha) \geq \frac{2}{n} \lambda_1[u](x) \left(\max_{1\leq i\leq n/2}\alpha_i\right).
\]
As a consequence, by the choice of $\theta$, we obtain
\[
    \inf_{\substack{\alpha \in \mathcal{A}_{n/2} \setminus \mathcal{A}_{n/2}(\theta)}} \mathcal{S}(\alpha) \geq \frac{2}{n} \lambda_1[u](x) \theta > \mathcal{L}[u](x).
\]
Taking into account that $\inf_{\alpha\in\mathcal{A}_{n/2}} \mathcal{S}(\alpha) \leq \mathcal{L}u(x)$, we obtain
\[
    \inf_{\substack{\alpha \in \mathcal{A}_{n/2} \setminus \mathcal{A}_{n/2}(\theta)}} \mathcal{S}(\alpha) > \inf_{\substack{\alpha \in \mathcal{A}_{n/2}}} \mathcal{S}(\alpha),
\]
from which \eqref{ineq:inf-realized-smaller-set} follows from the fact that
\[
    \inf_{\alpha\in\mathcal{A}_{n/2}}\mathcal{S}(\alpha) = \min\left\{\inf_{\substack{\alpha \in \mathcal{A}_{n/2} \setminus \mathcal{A}_{n/2}(\theta)}}\mathcal{S}(\alpha), \inf_{\substack{\alpha \in \mathcal{A}_{n/2}(\theta)}}\mathcal{S}(\alpha)\right\}.
\]

Since the infimum may be restricted to the compact set \(\mathcal A_{n/2}(\theta)\), it is attained.
\end{proof}

We also have a slightly different result, ensuring that an infimum is attained.

\begin{lemma}\label{lemma:restrict-set-alpha-2}
Let $u\colon \mathcal{X} \to \mathbb{R}$ and let $f\colon \mathcal X\setminus \mathcal O\to (0,\infty)$. For every $x\in\mathcal X\setminus\mathcal O$, there exists $\alpha^*\in\mathcal A_{n/2}$ such that
\[
    F[u](x)=\mathcal Q_u(x;\alpha^*).
\]
\end{lemma}

\begin{proof}
Fix \(x\in\mathcal X\setminus\mathcal O\), set \(m=n/2\), and write
\[
    h_i\coloneqq H_i[u](x),
    \qquad
    \rho\coloneqq f(x)^{1/m}.
\]
Let $t \in (-\infty,\min_i h_i)$. By the arithmetic–geometric mean inequality, we have
\[
    \sum_{i=1}^{m} \alpha_i(h_i - t) \geq m \bigg(\prod_{i=1}^m(h_i - t)\bigg)^{1/m},
\]
where we used that $\prod_{i=1}^m \alpha_i = 1$, and the equality holds for
\[
    \alpha_i = \frac{\bigg(\prod_{j=1}^m(h_j - t) \bigg)^{1/m}}{h_i - t}.
\]
Now we define the function
\[
    \phi(t) \coloneqq \prod_{i=1}^m(h_i - t),
\]
and observe that $\phi$ is strictly decreasing in the interval $(-\infty,\min_i h_i)$, and 
\[
    \lim_{t\to-\infty}\phi(t)=+\infty,
    \qquad
    \lim_{t\uparrow h_1}\phi(t)=0.
\]
By continuity, there exists a unique $t^*$ so that $\phi(t^*)=\rho^m$. If we select
\[
    \alpha_i^* = \frac{\rho}{h_i - t^*},
\]
we have
\[
    \prod_{i=1}^{m}\alpha_i^*
    =
    \frac{\rho^m}{\prod_{i=1}^{m}(h_i-t^*)}
    =1,
\]
so \(\alpha^*\in\mathcal A_m\). Then, \(F[u](x) = \mathcal{Q}_u(x;\alpha^*)\). Indeed, if we let $\alpha \in \mathcal{A}_{n/2}$, we write 
\begin{align*}
    \mathcal{Q}_u(x;\alpha) & = t^* + \frac{\sum_{i=1}^{m} \alpha_i (h_i-t^*) - m\rho}{\sum_{i=1}^{m} \alpha_i}\\
    & \geq  t^* + \frac{m \bigg(\prod_{i=1}^m(h_i - t^*)\bigg)^{1/m} - m\rho}{\sum_{i=1}^{m} \alpha_i} = t^*,
\end{align*}
where the last equality follows from the choice of $t^*$. Taking the infimum on both sides, we have
\[
    F[u](x) \geq t^*.
\]
Recalling the definition of $\alpha_i^*$, we have $\alpha_i^*(h_i-t^*) = \rho$, and so
\[
    \sum_{i=1}^{m} \alpha_i^*(h_i-t^*) = m\rho,
\]
but this means that $t^* = \mathcal{Q}_u(x;\alpha^*)$, and so 
\[
    F[u](x) \geq \mathcal{Q}_u(x;\alpha^*).
\]
The other inequality follows directly from the definition of \(F[u](x)\).
\end{proof}

Now, we present the result that provides an alternative way of writing the inhomogeneous Monge--Amp\`ere equation.

\begin{theorem}\label{theorem:equivalent-formulation}
Let \(f\colon\mathcal X\setminus\mathcal O\to(0,\infty)\).
A function \(u\colon\mathcal X\to\mathbb R\) satisfies
\begin{equation}\label{eq:closed-form-MA}
    u(x)=F[u](x),
    \qquad x\in\mathcal X\setminus\mathcal O,
\end{equation}
if and only if it is strictly graph-convex on
\(\mathcal X\setminus\mathcal O\) and satisfies
\[
    \mathcal M[u](x)=f(x),
    \qquad x\in\mathcal X\setminus\mathcal O.
\]
\end{theorem}

\begin{proof}
Fix $x \in \mathcal{X} \setminus \mathcal{O}$. First, we argue that $\mathcal{Q}_u$ is bounded from below in $\mathcal{A}_{n/2}$, which ensures $\inf_{\alpha \in \mathcal{A}_{n/2}} \mathcal{Q}_u(x;\alpha)$ is well-defined. Indeed, from \(\mathcal{M}[u](x) = f(x)\), we have
\[
    \frac{n}{2}f(x)^{\frac{2}{n}} \leq \sum_{i=1}^{n/2}\alpha_i \left(H_i[u](x)-u(x)\right), \quad \text{for every} \quad \alpha \in \mathcal{A}_{n/2}, 
\]
which implies that
\[
    u(x) \sum_{i=1}^{n/2}\alpha_i \leq \sum_{i=1}^{n/2}\alpha_i H_i[u](x) - \frac{n}{2}f(x)^{\frac{2}{n}}, \quad \text{for every} \quad \alpha \in \mathcal{A}_{n/2}.
\]
Since $\sum_{i=1}^{n/2}\alpha_i > 0$ for every $\alpha \in \mathcal{A}_{n/2}$, we obtain
\[
    u(x) \leq \mathcal{Q}_u(x;\alpha), \quad \text{for every} \quad \alpha \in \mathcal{A}_{n/2}.
\]
This implies that $\mathcal{Q}_u$ is bounded from below in $\mathcal{A}_{n/2}$, and taking the infimum on both sides gives 
\begin{equation}\label{eq:closed-form-MA.99}
    u(x) \leq F[u](x).
\end{equation}
To prove the reverse inequality, we proceed as follows: take a minimizing sequence in \(\mathcal{M}[u](x) = f(x)\), $(\alpha^k)_{k \in \mathbb{N}} \subseteq \mathcal{A}_{n/2}$ such that
\begin{equation}\label{eq:minimizing-sequence}
    \sum_{i=1}^{n/2}\alpha^k_i \left(H_i[u](x)-u(x)\right) \to \frac{n}{2}f(x)^{2/n} \quad \text{as} \quad k \to \infty.
\end{equation}
As a consequence of the arithmetic-geometric mean inequality, we have that
\begin{equation}\label{eq:bounded-from-below-min-sequence}
    \sum_{i=1}^{n/2} \alpha_i^k \geq \frac{n}{2},
\end{equation}
for every $k \in \mathbb{N}$. We can then write
\[
    F[u](x) \leq \mathcal{Q}_u(x;\alpha^k) = \frac{\sum_{i=1}^{n/2}\alpha^k_i \left(H_i[u](x) - u(x) \right) - \frac{n}{2}f(x)^{2/n}}{\sum_{i=1}^{n/2} \alpha_i^k} + u(x).
\]
Using \eqref{eq:minimizing-sequence} and \eqref{eq:bounded-from-below-min-sequence}, we can pass to the limit in the inequality above as $k \to \infty$ to obtain
\[
    F[u](x) \leq u(x),
\]
which proves \eqref{eq:closed-form-MA}.

Now we prove the converse statement. Assuming $u$ has the form \eqref{eq:closed-form-MA}, we have
\[
    u(x) \leq \mathcal{Q}_u(x;\alpha), \quad \text{for every} \quad \alpha \in \mathcal{A}_{n/2},
\]
which implies
\[
    \frac{n}{2}f(x)^{\frac{2}{n}} \leq \sum_{i=1}^{n/2}\alpha_i(H_i[u](x) - u(x)), \quad \text{for every} \quad \alpha \in \mathcal{A}_{n/2}.
\]
Taking the infimum on both sides gives us \(\mathcal{M}[u](x) \geq f(x)\). To prove the other inequality, we use Lemma \ref{lemma:restrict-set-alpha-2} to obtain $\alpha^*$ such that
\[
    \mathcal{Q}_u(x;\alpha^*) = u(x).
\]
We then write
\begin{align*}
    \frac n2\bigl(\mathcal M[u](x)\bigr)^{2/n} & \leq \sum_{i=1}^{n/2}\alpha^*_i \left(H_i[u](x)-u(x)\right)\\
                    & = \sum_{i=1}^{n/2}\alpha^*_i H_i[u](x) - u(x) \sum_{i=1}^{n/2}\alpha^*_i - \frac{n}{2}f(x)^{\frac{2}{n}} + \frac{n}{2}f(x)^{\frac{2}{n}}\\
                    & = \left(\sum_{i=1}^{n/2}\alpha^*_i\right) \left(\mathcal{Q}_u(x;\alpha^*) - u(x)\right) + \frac{n}{2}f(x)^{\frac{2}{n}}\\
                    & = \frac{n}{2}f(x)^{\frac{2}{n}},
\end{align*}
which implies \(\mathcal{M}[u](x) \leq f(x)\).
\end{proof}

\subsection{Comparison principle}

In this section, we prove the comparison principle for subsolutions and supersolutions to \eqref{eq:M-A}. We start with a result related to locality.

\begin{proposition}\label{prop:propagation-to-boundary}
Let \(u,v\colon\mathcal X\to \mathbb R\) satisfy \(u\le v\) in \(\mathcal X\) and
\[
    u(x_0)=v(x_0)
\]
for some \(x_0\in \mathcal X\). Assume moreover that
\[
    H_i[u](x_0)=H_i[v](x_0)\qquad \text{for every } i\in\{1,\dots,n/2\}.
\]
Then \(u=v\) on \(\mathcal N_{x_0}\).
\end{proposition}

\begin{proof}
Let \(n=|\mathcal N_{x_0}|\), and let
\[
    a_1\le\cdots\le a_n,
    \qquad
    b_1\le\cdots\le b_n
\]
be the nondecreasing rearrangements of
\[
    \{u(y):y\in\mathcal N_{x_0}\}
    \quad\text{and}\quad
    \{v(y):y\in\mathcal N_{x_0}\},
\]
respectively.

We first claim that
\[
    a_k\le b_k,\qquad k=1,\dots,n.
\]
Indeed, at least \(k\) neighbors satisfy \(v(y)\le b_k\). For each of
these neighbors, the assumption \(u\le v\) gives
\[
    u(y)\le v(y)\le b_k.
\]
Thus at least \(k\) of the neighbor values of \(u\) are at most \(b_k\),
which implies \(a_k\le b_k\).

By the order-statistic representation of \(H_i\),
\[
    a_{2i-1}+a_{2i}
    =
    2H_i[u](x_0)
    =
    2H_i[v](x_0)
    =
    b_{2i-1}+b_{2i},
    \qquad i=1,\dots,\frac n2.
\]
Since
\[
    a_{2i-1}\le b_{2i-1}
    \quad\text{and}\quad
    a_{2i}\le b_{2i},
\]
equality of the sums implies
\[
    a_{2i-1}=b_{2i-1},
    \qquad
    a_{2i}=b_{2i}.
\]
Consequently \(a_k=b_k\) for every \(k=1,\dots,n\), and hence
\[
    \sum_{y\in\mathcal N_{x_0}}u(y)
    =
    \sum_{y\in\mathcal N_{x_0}}v(y).
\]
Because \(u(y)\le v(y)\) at every neighbor, we have
\[
    \sum_{y\in\mathcal N_{x_0}}\bigl(v(y)-u(y)\bigr)=0,
\]
with every summand nonnegative. Therefore \(u(y)=v(y)\) for every
\(y\in\mathcal N_{x_0}\).
\end{proof}

Now we can prove the comparison principle.

\begin{theorem}\label{thm:comparison-principle}
Let \(f\colon\mathcal X\setminus\mathcal O
\to(0,\infty)\), and let \(u,v\colon\mathcal X\to\mathbb R\) satisfy
\begin{equation}\label{eq:bvp-sub}
    \begin{cases}
        \displaystyle
        u(x) \leq F[u](x),
        & x\in\mathcal X\setminus\mathcal O,\\[3ex]
        u(x)\leq g(x),
        & x\in\mathcal O,
    \end{cases}
\end{equation}
and
\begin{equation}\label{eq:bvp-super}
    \begin{cases}
        \displaystyle
        v(x) \geq F[v](x),
        & x\in\mathcal X\setminus\mathcal O,\\[3ex]
        v(x)\geq g(x),
        & x\in\mathcal O.
    \end{cases}
\end{equation}
Then
\[
    u\leq v
    \qquad\text{in }\mathcal X.
\]
\end{theorem}

\begin{proof}
Set \( m\coloneqq n/2\).  The monotonicity of the operator \(F\) follows from the monotonicity of each \(H_i\), and moreover
\begin{equation}\label{eq:F-translation}
    F[z+c](x)=F[z](x)+c.
\end{equation}
We first prove that
\begin{equation}\label{eq:F-strictly-below-H1}
    F[z](x)<H_1[z](x),
\end{equation}
for every \(z\colon\mathcal X\to\mathbb R\) and every
\(x\in\mathcal X\setminus\mathcal O\). Indeed, write
\[
    h_i\coloneqq H_i[z](x),
    \qquad i=1,\dots,m.
\]
Since \(h_1\leq\cdots\leq h_m\) and \(f(x)>0\), there exists a unique
\(t<h_1\) such that
\[
    \prod_{i=1}^{m}(h_i-t)=f(x).
\]
Define
\[
    \alpha_i
    \coloneqq
    \frac{f(x)^{1/m}}{h_i-t},
    \qquad i=1,\dots,m.
\]
Then \(\alpha_i>0\) and
\[
    \prod_{i=1}^{m}\alpha_i
    =
    \frac{f(x)}
    {\prod_{i=1}^{m}(h_i-t)}
    =1,
\]
so \(\alpha\in\mathcal A_m\). Moreover,
\[
    \alpha_i(h_i-t)=f(x)^{1/m},
\]
and therefore
\[
\mathcal{Q}_z(x;\alpha) = t+ \frac{\sum_{i=1}^{m}\alpha_i(h_i-t) -m f(x)^{1/m}}{\sum_{i=1}^{m}\alpha_i}=t.
\]
It follows that
\[
    F[z](x)\leq t<h_1=H_1[z](x),
\]
which proves \eqref{eq:F-strictly-below-H1}.

Suppose, seeking a contradiction, that
\[
    M
    \coloneqq
    \max_{x\in\mathcal X}\bigl(u(x)-v(x)\bigr)>0.
\]
Define \(A \coloneqq \left\{x\in\mathcal X \colon u(x)-v(x)=M\right\}\) and observe that since \(u\leq g\leq v\) on \(\mathcal O\), we have \(A\subseteq\mathcal X\setminus\mathcal O\). Setting \(w\coloneqq v+M\), we obtain \(u\leq w\) in \(\mathcal X\), and
\[
    u(x)=w(x)
    \qquad\text{for every }x\in A.
\]

Fix \(x\in A\). Using the subsolution inequality for \(u\), the
monotonicity and translation invariance of \(F\), and the supersolution inequality for \(v\), we obtain
\[
\begin{aligned}
    u(x) \leq F[u](x) \leq F[w](x) = F[v](x)+M \leq v(x)+M=u(x).
\end{aligned}
\]
In particular
\begin{equation}\label{eq:F-contact}
    F[u](x)=F[w](x).
\end{equation}
By Lemma \ref{lemma:restrict-set-alpha-2}, there exists
\(\alpha^*\in\mathcal A_m\) such that \(F[w](x)=\mathcal{Q}_w(x;\alpha^*)\). Using \eqref{eq:F-contact} and the definition of \(F[u]\), we get
\[
    \mathcal{Q}_w(x;\alpha^*)
    =
    F[w](x)
    =
    F[u](x)
    \leq
    \mathcal{Q}_u(x;\alpha^*).
\]
On the other hand, \(u\leq w\) implies \(\mathcal{Q}_u(x;\alpha^*)\leq \mathcal{Q}_w(x;\alpha^*)\), and hence
\[
    \mathcal{Q}_u(x;\alpha^*)=\mathcal{Q}_w(x;\alpha^*).
\]
Consequently,
\[
    0
    =
    \mathcal{Q}_w(x;\alpha^*)-\mathcal{Q}_u(x;\alpha^*)
    =
    \frac{
        \displaystyle
        \sum_{i=1}^{m}\alpha_i^*
        \bigl(H_i[w](x)-H_i[u](x)\bigr)
    }{
        \displaystyle\sum_{i=1}^{m}\alpha_i^*
    }.
\]
Every term in the numerator is nonnegative, and every
\(\alpha_i^*\) is positive. Therefore,
\[
    H_i[u](x)=H_i[w](x),
    \qquad i=1,\dots,m.
\]

Since \(u\leq w\), \(u(x)=w(x)\), and all the corresponding \(H_i\) agree at \(x\), Proposition \ref{prop:propagation-to-boundary} gives \(u=w\) on \(\mathcal N_x\), and so
\begin{equation}\label{eq:contact-set-closed}
    \mathcal N_x\subseteq A
    \qquad\text{for every }x\in A.
\end{equation}
Since \(A\) is finite and nonempty, choose
\(\widehat x\in A\) such that
\[
    u(\widehat x)=\max_{x\in A}u(x).
\]
By \eqref{eq:contact-set-closed}, every neighbor of \(\widehat x\)
belongs to \(A\). Therefore,
\[
    u(y)\leq u(\widehat x)
    \qquad\text{for every }y\in\mathcal N_{\widehat x},
\]
which implies
\[
    H_1[u](\widehat x)\leq u(\widehat x).
\]
However, the subsolution inequality and
\eqref{eq:F-strictly-below-H1} give
\[
    u(\widehat x)
    \leq
    F[u](\widehat x)
    <
    H_1[u](\widehat x)
    \leq
    u(\widehat x),
\]
which is impossible.

We conclude that
\[
    \max_{x\in\mathcal X}\bigl(u(x)-v(x)\bigr)\leq0,
\]
and therefore \(u\leq v\) in \(\mathcal X\).
\end{proof}

As a consequence of the comparison principle, we obtain uniqueness of solutions.

\begin{corollary}\label{cor:uniqueness-solutions}
A solution $u\colon \mathcal{X} \to \mathbb{R}$ to \eqref{eq:M-A}, whenever it exists, is unique.
\end{corollary}

\begin{proof}
Let \(u\) and \(v\) be two solutions. By Theorem \ref{theorem:equivalent-formulation}, $u$ and $v$ have the form \eqref{eq:closed-form-MA}. We can then use Theorem \ref{thm:comparison-principle} to obtain $u = v$.
\end{proof}

\subsection{Perron's method}

In this section, we prove the existence of solutions to \eqref{eq:M-A} via the Perron method, whose main step consists of proving that there exists at least one subsolution. Since being a solution to the equation in \eqref{eq:M-A} is equivalent, via Theorem \ref{theorem:equivalent-formulation}, to solving \eqref{eq:closed-form-MA}, we deal with the latter.

\medskip

We consider the set
\[
\mathcal{S}_{\mathcal{M}} \coloneqq
\left\{
v \colon \mathcal{X} \to \mathbb{R}
\;\middle|\;
\begin{array}{ll}
\displaystyle
v(x) \leq F[v](x),
& x \in \mathcal{X} \setminus \mathcal{O},
\\[1.2ex]
v(x) \leq g(x),
& x \in \mathcal{O}.
\end{array}
\right\}
\]
of all subsolutions to the Monge--Amp\`ere Dirichlet problem. 

The following theorem guarantees existence under certain conditions.

\begin{theorem}\label{thm:large-sub}
Assume $\mathcal{S}_{\mathcal{M}}$ is not empty, and that for each $x \in \mathcal{X}$, there exists $M(x)$ such that
\[
    v(x) \leq M(x) \quad \text{for all} \quad v \in \mathcal{S}_{\mathcal{M}}.
\]
Then, the function $u\colon \mathcal{X} \to \mathbb{R}$ defined by
\[
    u(x) \coloneqq \sup_{v \in \mathcal{S}_{\mathcal{M}}}v(x)
\]
is the unique solution to \eqref{eq:M-A}.
\end{theorem}

\begin{proof}
Since \(\mathcal S_{\mathcal M}\) is nonempty and pointwise bounded from above, the function
\[
    u(x) \coloneqq \sup_{v\in\mathcal S_{\mathcal M}}v(x)
\]
is finite for every \(x\in\mathcal X\). We first show that \(u\) is a subsolution. Let \(v\in\mathcal S_{\mathcal M}\). Since \(v\le u\), monotonicity gives
\[
    F[v](x)\le F[u](x),
    \qquad x\in\mathcal X\setminus\mathcal O.
\]
Then, using that \(v\) is a subsolution, we get
\[
    v(x)\le F[v](x)\le F[u](x),
    \qquad x\in\mathcal X\setminus\mathcal O.
\]
Taking the supremum over \(v\in\mathcal S_{\mathcal M}\), we obtain
\[
    u(x)\le F[u](x),
    \qquad x\in\mathcal X\setminus\mathcal O.
\]
Moreover, since every \(v\in\mathcal S_{\mathcal M}\) satisfies \(v\le g\)
on \(\mathcal O\), we have
\[
    u(x)\le g(x),
    \qquad x\in\mathcal O.
\]
Hence \(u\in\mathcal S_{\mathcal M}\).

We now show that \(u\) attains the boundary data. Define
\[
    u^g(x) \coloneqq 
    \begin{cases}
        u(x), & x\in\mathcal X\setminus\mathcal O,\\
        g(x), & x\in\mathcal O.
    \end{cases}
\]
Then \(u^g\ge u\) on \(\mathcal X\). By monotonicity of \(F\), for \(x\in\mathcal X\setminus\mathcal O\),
\[
    u^g(x)
    =
    u(x)
    \le
    F[u](x)
    \le
    F[u^g](x).
\]
Also \(u^g=g\) on \(\mathcal O\). Thus \(u^g\in\mathcal S_{\mathcal M}\). Since \(u\) is the supremum of all subsolutions,
\[
    u^g(x)\le u(x),
    \qquad x\in\mathcal X.
\]
But \(u\le u^g\), and therefore \(u=u^g\). Hence
\[
    u(x)=g(x),
    \qquad x\in\mathcal O.
\]
It remains to show that \(u\) satisfies the equation in
\(\mathcal X\setminus\mathcal O\). Suppose, seeking a contradiction, that
there exists \(\tilde x\in\mathcal X\setminus\mathcal O\) such that
\[
    0<\delta \eqqcolon F[u](\tilde x) - u(\tilde x).
\]
Define \(\tilde u \colon \mathcal X\to\mathbb R\) by
\[
    \tilde u(x) \coloneqq
    \begin{cases}
        u(x), & x\neq \tilde x,\\
        u(\tilde x)+\delta, & x=\tilde x.
    \end{cases}
\]
Then \(\tilde u\ge u\) on \(\mathcal X\). We claim that \(\tilde u\) is a
subsolution. First, at \(x=\tilde x\), monotonicity gives
\[
    F[\tilde u](\tilde x)
    \ge
    F[u](\tilde x)
    =
    u(\tilde x)+\delta
    =
    \tilde u(\tilde x).
\]
Second, if \(x\in\mathcal X\setminus\mathcal O\) and \(x\neq\tilde x\), then
\[
    F[\tilde u](x)
    \ge
    F[u](x)
    \ge
    u(x)
    =
    \tilde u(x).
\]
Finally, since \(\tilde x\notin\mathcal O\), we have \(\tilde u=u=g\) on \(\mathcal O\). Therefore \(\tilde u\in\mathcal S_{\mathcal M}\). By the definition of \(u\) as the supremum of all subsolutions,
\[
    u(\tilde x)
    \ge
    \tilde u(\tilde x)
    =
    u(\tilde x)+\delta,
\]
which is impossible. Hence
\[
    u(x)=F[u](x),
    \qquad x\in\mathcal X\setminus\mathcal O.
\]
Together with \(u=g\) on \(\mathcal O\), this proves that \(u\) solves the
Bellman Dirichlet problem.

Uniqueness follows from the comparison principle: if \(w\) is another solution, then \(u\) and \(w\) are both subsolutions and supersolutions, so applying the comparison principle twice gives \(u\le w\) and \(w\le u\). Therefore \(u=w\).
\end{proof}

In the previous result, we have proved that if $\mathcal{S}_{\mathcal{M}}$ is nonempty, and it is bounded for each $x \in \mathcal{X}$, then the largest subsolution is the actual solution to the problem. We next identify graph conditions ensuring these two assumptions. We begin with the following remark.

\begin{remark}
Let $w\colon \mathcal{X} \to \mathbb{R}$ and $f\colon \mathcal{X} \to \mathbb{R}$ be nonnegative. Recalling that if $\alpha \in \mathcal{A}_{n/2}$, we obtain
\[
    \sum_{i=1}^{n/2}\alpha_i \geq \frac{n}{2}.
\]
Therefore, using that $H_i[w] \geq H_1[w]$ and the nonnegativity of $f$, we have
\begin{align*}
    F[w](x) & \geq \inf_{\alpha\in\mathcal{A}_{n/2}}\left(H_1[w](x) - \frac{\frac{n}{2}f(x)^{2/n}}{\sum_{i=1}^{n/2}\alpha_i}\right)\\
        &\geq  H_1[w](x)-f(x)^{2/n}.
\end{align*}
On the other hand, using that $H_i[w] \leq H_{n/2}[w]$ and the nonnegativity of $f$, we have
\[
    F[w](x)  \leq H_{n/2}[w](x).
\]
In particular, if $w_1$ and $w_{n/2}$ satisfy
\[
w_1 (x) \leq H_{1}[w_1](x)-f(x)^{2/n} \quad \text{and} \quad w_{n/2} (x) \geq H_{n/2}[w_{n/2}] (x), \quad x\in \mathcal{X}\setminus\mathcal{O},
\]
with $$w_1 (x) \leq g (x) \leq w_{n/2} (x) \quad x \in  \mathcal{O},$$ then, $w_1 \in \mathcal{S}_{\mathcal{M}}$ and by Theorem \ref{thm:comparison-principle}, we have
\[
     w (x) \leq w_{n/2} (x) \quad \forall x \in \mathcal{X}, \quad \text{for all} \quad w \in \mathcal{S}_{\mathcal{M}}.
\]
\end{remark}

\subsection{Existence on 1-degenerate graphs}

In view of the preceding remark, the assumptions of Theorem \ref{thm:large-sub} are satisfied whenever we can ensure the existence of a subsolution and of a supersolution to the extremal problems with $H_1$ and $H_{n/2}$. Observe that if we take $w_{n/2} = \|g\|_{L^\infty(\mathcal{O})}$, then
\[
    w_{n/2}(x) = H_{n/2} [w_{n/2}](x) \quad \text{for} \quad x \in \mathcal{X} \setminus \mathcal{O},
\]
and $w_{n/2} \geq g$ on $\mathcal{O}$. In particular, we get
\[
    w \leq \|g\|_{L^\infty(\mathcal{O})} \quad \text{for all} \quad w \in \mathcal{S}_{\mathcal{M}}.
\]

Existence for the problem involving \(H_1\) is more delicate. Without additional assumptions, a solution may fail to exist, as the following example shows.

\begin{example}
Let $\mathcal{X} = \{1,2,3\}$, fully connected, $\mathcal{O} = \emptyset$ and $f \equiv 1$. By vacuity, the boundary condition is attained. If there were such $w_1$, we would have
\begin{align*}
    w_1(1) &\leq \frac{(w_1(2) + w_1(3))}{2} - 1\\
    w_1(2) &\leq \frac{(w_1(1) + w_1(3))}{2} - 1\\
    w_1(3) &\leq \frac{(w_1(1) + w_1(2))}{2} - 1.
\end{align*}
Then,
\[
    w_1(1) + w_1(2) + w_1(3) \leq w_1(1) + w_1(2) + w_1(3) - 3,
\]
which is a contradiction. In fact, more generally, a barrier cannot exist if there is a nonempty set $A \subset \mathcal{X} \setminus \mathcal{O}$ such that
\[
    \mathcal{N}_x \subset A, \quad \text{for every} \quad x \in A.
\]
\end{example}

The previous example shows that one cannot expect solutions to exist in full generality. To guarantee their existence, one must impose enough additional structure on the graph to ensure the existence of a solution to the problem
\[
\begin{cases}
    w_1 (x) \leq H_{1}[w_1](x)-f(x)^{2/n} &x \in \mathcal{X}\setminus\mathcal{O},\\
    w_1 (x) \leq g (x) & x\in  \mathcal{O}.
\end{cases}
\]

\medskip

We begin by noticing that it is sufficient to establish existence for the following simpler problem
\[
    b(x) \geq H_{n/2}[b](x)+1 \quad \text{for} \quad x \in \mathcal{X}\setminus \mathcal{O},
\]
and $b \geq 0$ in $\mathcal{O}$. Indeed, if such $b$ exists, then we take
\[
    w_1(x) \coloneqq -\lambda b(x) - M , \quad \text{with} \quad \lambda \coloneqq \sup_{x \in \mathcal{X} \setminus \mathcal{O}}f(x)^{\frac{2}{n}}, \quad M \coloneqq \|g\|_{L^\infty(\mathcal{O})},
\]
which satisfies $w_1 \leq g$ in $\mathcal{O}$ and
\[
    H_1[w_1](x) = -\lambda H_{n/2} [b](x) - M \geq -\lambda (b(x) - 1) - M \geq w_1(x) + f(x)^{2/n},
\]
for $x \in \mathcal{X} \setminus \mathcal{O}$.

It turns out that the existence of such barriers imposes a $1$-degeneracy structure on the graph, as in the following result.

\begin{theorem}\label{thm:plucking-condition}
Let \(\mathcal G\) be a finite simple undirected graph with vertex
set \(\mathcal X\), and let \(\mathcal O\subset\mathcal X\). Assume
that every \(x\in\mathcal X\setminus\mathcal O\) has at least two
neighbors. Consider the system
\begin{equation}\label{eq:system}
\begin{cases}
b(x)\ge \dfrac{b(y_1)+b(y_2)}{2}+1, & x\in \mathcal{X}\setminus \mathcal{O},\\[1ex]
b(x)\ge 0, & x\in \mathcal{O},
\end{cases}
\end{equation}
where $y_1,y_2\in\mathcal N_x$ are such that
$$b(y_1)\ge b(y_2)\ge b(y),\quad \forall y\in\mathcal{N}_x\setminus\{y_1,y_2\}.$$

Then there exists a function $b:\mathcal{X}\to\mathbb{R}$ satisfying
\eqref{eq:system} if and only if
\begin{equation}\label{eq:plucking-condition}
    \boxed{
    \forall\,\emptyset\neq S\subset \mathcal{X}\setminus \mathcal{O},
    \quad
    \exists x\in S \text{ such that }
    |\mathcal{N}_x\cap S|\le 1 .
    }
\end{equation}
\end{theorem}

\begin{proof}

\noindent\textbf{Necessity.} Assume that a function $b$ satisfying \eqref{eq:system} exists. Let $\emptyset\neq S\subset \mathcal{X}\setminus\mathcal{O}$, and choose $x\in S$ such that
\[
    b(x)=\min_{z\in S} b(z).
\]
We claim that $|\mathcal{N}_x\cap S|\le 1$. Suppose, to the contrary, that $|\mathcal{N}_x\cap S|\ge 2$. Then there exist two distinct neighbors $z_1,z_2\in\mathcal{N}_x\cap S$. Since $x$ minimizes $b$ on $S$, we have
\[
    b(z_1)\ge b(x),
    \qquad
    b(z_2)\ge b(x).
\]
In particular, since $y_1$ and $y_2$ are chosen so that $b(y_1)$ and $b(y_2)$ are the two largest values of $b$ on $\mathcal{N}_x$, it follows that
\[
    b(y_1)\ge b(x),
    \qquad
    b(y_2)\ge b(x).
\]
Using \eqref{eq:system} at $x$, we obtain
\[
    b(x)
    \ge
    1+\frac{b(y_1)+b(y_2)}{2}
    \ge
    1+b(x),
\]
which is impossible. Hence $|\mathcal{N}_x\cap S|\le 1$. Since $S$ was arbitrary, \eqref{eq:plucking-condition} follows.

\noindent\textbf{Sufficiency.} Assume that \eqref{eq:plucking-condition} holds. Set
\[
     \mathcal{U} \coloneqq \mathcal{X}\setminus\mathcal{O}.
\]
If $\mathcal{U}=\emptyset$, then $b\equiv 0$ satisfies \eqref{eq:system} by vacuity. Hence, we may assume $\mathcal{U}\neq\emptyset$. Define inductively
\[
    S_0 \coloneqq \mathcal{U},
    \qquad
    A_k \coloneqq \{x\in S_k \colon \;|\mathcal{N}_x\cap S_k|\le 1\},
    \qquad
    S_{k+1} \coloneqq S_k\setminus A_k.
\]
By \eqref{eq:plucking-condition}, $A_k$ is nonempty whenever $S_k$ is nonempty. Since $\mathcal{X}$ is finite, this process ends after finitely many steps (see Figure \ref{fig:peeling-process} below).

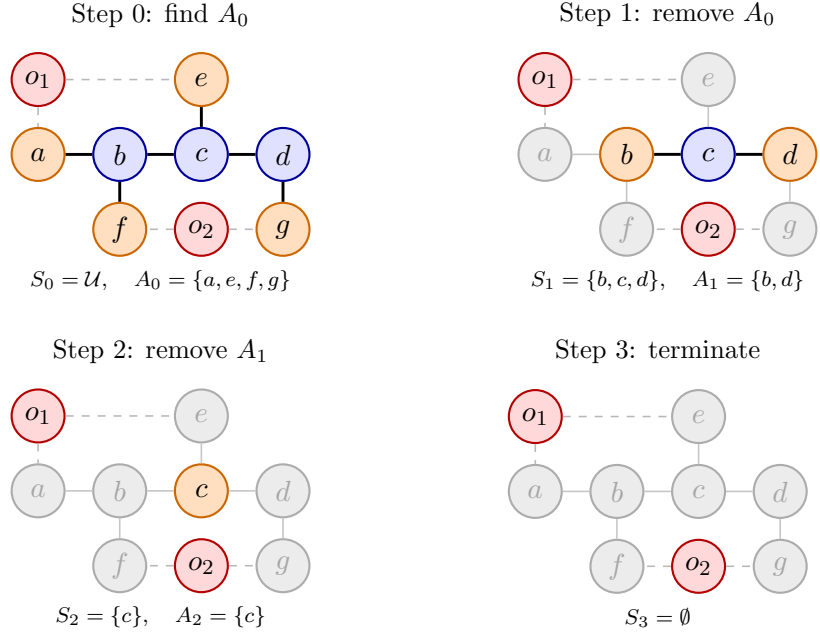
\begin{figure}[ht]
\centering

\input{TikZ/PeelingFigure.tikz}

\caption{A peeling sequence for the sets
$S_0=\mathcal{U}$, $A_k=\{x\in S_k:\ |\mathcal N_x\cap S_k|\le 1\}$, and
$S_{k+1}=S_k\setminus A_k$. Orange vertices are removed at the current step,
blue vertices remain, gray vertices have already been removed, and red
vertices belong to $\mathcal O$.}
\label{fig:peeling-process}
\end{figure}

Thus, for some $m\ge 0$, we obtain disjoint layers
\[
    A_0,A_1,\dots,A_m
\]
such that
\[
    \mathcal{U}=A_0\cup A_1\cup\cdots\cup A_m.
\]
We now choose values in each layer large enough to dominate the possible contribution of one neighbor lying in a later layer. Set
\[
    B_{-1} \coloneqq 0,
    \qquad
    B_k \coloneqq 2^{m+2}-2^{m+1-k},
    \qquad 0\le k\le m .
\]
In particular, $B_0<B_1<\cdots<B_m$, and one checks that
\[
    B_k
    =
    1+\frac{B_m+B_{k-1}}{2},
    \qquad 0\le k\le m .
\]
Define $b \colon \mathcal{X}\to\mathbb{R}$ by
\[
    b(x)\coloneqq 
    \begin{cases}
        0, & x\in\mathcal{O},\\
        B_k, & x\in A_k.
    \end{cases}
\]
Clearly $b\ge 0$ on $\mathcal{O}$.

It remains to verify \eqref{eq:system} on $\mathcal{U}$. Fix $x\in A_k$. At the stage when $A_k$ is removed, we have
\[
    |\mathcal{N}_x\cap S_k|\le 1.
\]
Therefore, among the neighbors of $x$, at most one belongs to
\[
    S_k=A_k\cup A_{k+1}\cup\cdots\cup A_m .
\]
At this possible exceptional neighbor
the value of $b$ is at most $B_m$.

All the remaining neighbors belong to
\[
    \mathcal{O}\cup A_0\cup\cdots\cup A_{k-1},
\]
and hence the value of $b$ at those neighbors is at most $B_{k-1}$. 

Therefore, the two largest neighbor values satisfy
\[
    \frac{b(y_1)+b(y_2)}{2}
    \le
    \frac{B_m+B_{k-1}}{2}.
\]
Using the definition of $B_k$, we obtain
\[
    \frac{b(y_1)+b(y_2)}{2}
    \le
    B_k-1
    =
    b(x)-1.
\]
Thus
\[
    b(x)\ge \frac{b(y_1)+b(y_2)}{2}+1.
\]
Since $x\in \mathcal{U}$ was arbitrary, \eqref{eq:system} holds on $\mathcal{X}\setminus\mathcal{O}$. This proves sufficiency. 
\end{proof}

\begin{lemma}\label{lem:plucking-forest}
Condition \eqref{eq:plucking-condition} holds if and only if the
subgraph induced by $\mathcal X\setminus\mathcal O$ is a forest.
\end{lemma}

\begin{proof}
If \eqref{eq:plucking-condition} holds, the induced subgraph cannot
contain a cycle: the vertex set $S$ of a cycle would satisfy
$|\mathcal N_x\cap S|\ge2$ for every $x\in S$.

Conversely, suppose that the subgraph induced by $\mathcal X\setminus\mathcal O$ is a forest. For every nonempty
$S\subseteq\mathcal X\setminus\mathcal O$, the subgraph induced by
$S$ is a nonempty forest and therefore contains a vertex of degree at
most one. Hence there exists $x\in S$ such that
\[
    |\mathcal N_x\cap S|\le1,
\]
which is \eqref{eq:plucking-condition}.
\end{proof}

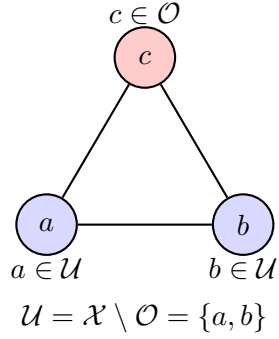
\begin{figure}[ht]
\centering
\input{TikZ/SimpleEx.tikz}
\caption{The full graph contains a cycle, so it is not a tree. However, the induced unlabeled graph on $\mathcal{U}=\{a,b\}$ is a single edge, hence a tree. Therefore, condition \eqref{eq:plucking-condition} can still hold.}
\label{fig:ex-cycle-with-condition}
\end{figure}

\section{Woven forests}

In light of Theorem \ref{thm:plucking-condition}, we introduce a family of graphs satisfying condition \eqref{eq:plucking-condition}. Throughout, all graphs are assumed to be simple and undirected, and $\operatorname{deg}_{\mathcal G}(x)$ denotes the degree of a vertex $x$ in a graph $\mathcal G$.

Let $\mathcal F$ be a graph on a vertex set $\mathcal U$, and let $\mathcal O$ be a set disjoint from $\mathcal U$. We augment $\mathcal F$ by adjoining vertices from $\mathcal O$ to compensate, as far as possible, for differences in the degrees of the vertices of $\mathcal U$.

\begin{definition}[Degree deficiency]
Let
\[
\Delta(\mathcal F)
\coloneqq
\max_{z\in\mathcal U}\operatorname{deg}_{\mathcal F}(z).
\]
The degree deficiency of $x\in\mathcal U$ in $\mathcal F$ is
\[
\operatorname{def}_{\mathcal F}(x)
\coloneqq
\Delta(\mathcal F)-\operatorname{deg}_{\mathcal F}(x).
\]
\end{definition}

\begin{definition}[Woven graph]
Let $\mathcal F$ be a graph on $\mathcal U$, and let $\mathcal O$ be disjoint from $\mathcal U$. A graph $\mathcal G$ on
\[
\mathcal X\coloneqq\mathcal U\cup\mathcal O
\]
is called an $\mathcal O$-woven graph over $\mathcal F$ if the subgraph of $\mathcal G$ induced by $\mathcal U$ is $\mathcal F$ and every $x\in\mathcal U$ has exactly
\[
\min\bigl\{
\operatorname{def}_{\mathcal F}(x),
|\mathcal O|
\bigr\}
\]
neighbors in $\mathcal O$. If $\mathcal F$ is a forest, we call $\mathcal G$ an $\mathcal O$-woven forest over $\mathcal F$.
\end{definition}

\begin{remark}
Every $\mathcal O$-woven forest $\mathcal{G}$ over $\mathcal F$ satisfies
\eqref{eq:plucking-condition}, because the subgraph induced by
$\mathcal U$ is the forest $\mathcal F$. Moreover, if
\[
    |\mathcal O|
    \ge
    \max_{x\in\mathcal U}\operatorname{def}_{\mathcal F}(x),
\]
then every \(x\in\mathcal U\) has degree \(\Delta(\mathcal F)\) in
\(\mathcal G\). Thus, provided that \(\Delta(\mathcal F)=n\) and the above inequality holds, $\mathcal{G}$ is an admissible graph for our Monge-Amp\`ere problem.
\end{remark}

Now, we state an algorithm for constructing woven graphs in $\mathbb{R}^d$ using nearest neighbors (cf. Figure \ref{fig:weave-ex}). Let $\mathcal{U}$ be a finite subset of $\mathbb{R}^d$, let \(\mathcal{F}=(\mathcal{U},E)\) be a graph, and let \(\mathcal{O}\subset \mathbb{R}^d\) be a finite set of auxiliary vertices with at least $\Delta(\mathcal F)$ elements. We wish to augment \(\mathcal{F}\) so that each vertex of \(\mathcal{U}\) attains degree \(\Delta(\mathcal F)\).

\begin{enumerate}
    \item For each \(x_i\in \mathcal{U}\), compute its degree deficiency
    \[
    \delta_i=\max\{0,\,\operatorname{def}_{\mathcal F}(x_i)\}.
    \]

    \item For each \(x_i\in \mathcal{U}\), determine its
    \[
    \max_{1\le i\le |\mathcal{U}|}\delta_i
    \]
    nearest neighbors in \(\mathcal{O}\).

    \item For each \(x_i\in \mathcal{U}\), add edges connecting \(x_i\) to its first
    \(\delta_i\) nearest neighbors in \(\mathcal{O}\).

    \item Let \(E_\mathcal{O}\) denote the collection of all added edges.

    \item Return the completed graph
    \[
    \mathcal{G}=(\mathcal{U}\cup\mathcal{O},\;E\cup E_\mathcal{O}).
    \]
\end{enumerate}

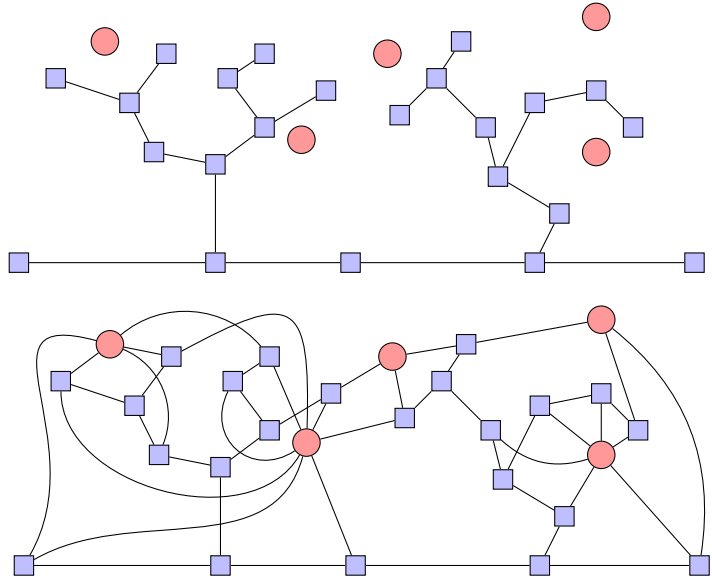
\begin{figure}[ht]
    \centering
    \quad\input{TikZ/Forest.tikz}
    
    \medskip
    
    \input{TikZ/WeavedForest.tikz}
    \caption{A construction of a woven forest.}
    \label{fig:weave-ex}
\end{figure}

\section{The homogeneous case}

The Bellman formulation \eqref{eq:closed-form-MA} is useful for the inhomogeneous Dirichlet problem, but the corresponding existence theory requires additional graph-theoretic assumptions. 
Now, we turn our attention to the homogeneous problem
\[
\begin{cases}
    \mathcal M [u](x)=0, & x\in \mathcal X\setminus \mathcal O,\\
    u(x)=g(x), & x\in \mathcal O.
\end{cases}
\]
For graph-convex functions, the equation
$\mathcal M [u](x)=0$ is equivalent to 
\[
    \lambda_1[u](x)=0.
\]
Using the identity
\[
    \lambda_1[u](x)=H_1[u](x)-u(x),
\]
the homogeneous Dirichlet problem becomes
\begin{equation}\label{eq:H-M-A}
    \begin{cases}
        u(x)=H_1[u](x),
        & x\in\mathcal X\setminus\mathcal O,\\
        u(x)=g(x),
        & x\in\mathcal O.
    \end{cases}
\end{equation}

We shall use the following mild reachability condition. For every nonemp\-ty set \(A\subset \mathcal X\setminus\mathcal O\), there exists \(x\in A\) such that
\begin{equation}\label{eq:H1-reachability}
    |\mathcal N_x\cap A|\le 1.
\end{equation}
Equivalently, no nonempty subset of the unlabeled graph is closed under the operation of retaining at least two neighbors of each of its vertices. This condition is the graph-theoretic mechanism that allows contact information for \(H_1\) to propagate to the boundary.

\subsection{Comparison principle}

We begin with a simple contact lemma adap\-ted to the operator \(H_1\).

\begin{lemma}\label{lem:H1-contact-propagation}
Let \(p,q\colon\mathcal X\to\mathbb R\) satisfy \(p\le q\) in \(\mathcal X\). If
\[
    H_1[p](x)=H_1[q](x)
\]
for some \(x\in\mathcal X\setminus\mathcal O\), then there exist two distinct neighbors \(y,z\in\mathcal N_x\) such that
\[
    p(y)=q(y)
    \qquad\text{and}\qquad
    p(z)=q(z).
\]
\end{lemma}

\begin{proof}
Choose distinct neighbors \(y,z\in\mathcal N_x\) such that
\[
    H_1[q](x)=\frac{q(y)+q(z)}{2}.
\]
Since \(p\le q\), we have
\[
    H_1[p](x)
    \le
    \frac{p(y)+p(z)}{2}
    \le
    \frac{q(y)+q(z)}{2}
    =
    H_1[q](x).
\]
The assumption \(H_1[p](x)=H_1[q](x)\) forces equality throughout. Hence
\[
    p(y)+p(z)=q(y)+q(z).
\]
Since \(p(y)\le q(y)\) and \(p(z)\le q(z)\), both inequalities must be equalities.
\end{proof}

\begin{theorem}\label{thm:homo-comparison-principle}
Assume \eqref{eq:H1-reachability}. Let \(u,v\colon\mathcal X\to\mathbb R\) satisfy
\[
    \begin{cases}
        u(x)\le H_1[u](x),
        & x\in\mathcal X\setminus\mathcal O,\\
        u(x)\le g(x),
        & x\in\mathcal O,
    \end{cases}
    \qquad
    \begin{cases}
        v(x)\ge H_1[v](x),
        & x\in\mathcal X\setminus\mathcal O,\\
        v(x)\ge g(x),
        & x\in\mathcal O.
    \end{cases}
\]
Then, \(u\le v\) in \(\mathcal X\).
\end{theorem}

\begin{proof}
Suppose, by contradiction, that
\[
    M\coloneqq \max_{x\in\mathcal X}(u(x)-v(x))>0.
\]
Let
\[
    A\coloneqq
    \{x\in\mathcal X \colon \; u(x)-v(x)=M\}.
\]
Since \(u\le g\le v\) on \(\mathcal O\), we have
\[
    A\subset \mathcal X\setminus\mathcal O.
\]
Fix \(x\in A\), and set \(w\coloneqq v+M\). Then \(u\le w\) in \(\mathcal X\) and \(u(x)=w(x)\). By the monotonicity and translation invariance of \(H_1\),
\[
    H_1[u](x)\le H_1[w](x)=H_1[v](x)+M.
\]
Using the subsolution inequality for \(u\) and the supersolution inequality for \(v\), we obtain
\[
    u(x)
    \le
    H_1[u](x)
    \le
    H_1[w](x)
    =
    H_1[v](x)+M
    \le
    v(x)+M
    =
    u(x).
\]
Therefore
\[
    H_1[u](x)=H_1[w](x).
\]
By Lemma \ref{lem:H1-contact-propagation}, there exist distinct neighbors \(y,z\in\mathcal N_x\) such that
\[
    u(y)=w(y)
    \qquad\text{and}\qquad
    u(z)=w(z).
\]
Equivalently,
\[
    u(y)-v(y)=M,
    \qquad
    u(z)-v(z)=M.
\]
Hence \(y,z\in A\), and so
\[
    |\mathcal N_x\cap A|\ge 2
    \qquad\text{for every }x\in A.
\]
This contradicts \eqref{eq:H1-reachability}, applied to the nonempty set \(A\). Thus \(M\le0\), and therefore \(u\le v\) in \(\mathcal X\).
\end{proof}

\begin{corollary}\label{cor:homo-uniqueness}
Assume \eqref{eq:H1-reachability}. Then a solution
\(u\colon\mathcal X\to\mathbb R\) to \eqref{eq:H-M-A}, whenever it exists, is unique.
\end{corollary}

\begin{proof}
If \(u\) and \(v\) are two solutions of \eqref{eq:H-M-A}, then \(u\) is a subsolution and \(v\) is a supersolution. The comparison principle gives \(u\le v\). Reversing the roles of \(u\) and \(v\), we obtain \(v\le u\). Hence \(u=v\) in \(\mathcal X\).
\end{proof}

\subsection{Existence of solutions}

We now prove existence by Perron's method. Define the class of subsolutions
\[
\mathcal S_H
\coloneqq
\left\{
v\colon\mathcal X\to\mathbb R
\;\middle|\;
\begin{array}{ll}
v(x)\le H_1[v](x),
& x\in\mathcal X\setminus\mathcal O,\\[0.8ex]
v(x)=g(x),
& x\in\mathcal O.
\end{array}
\right\}.
\]
The equality condition on \(\mathcal O\) is important: it ensures that the Perron envelope attains the prescribed boundary data.

Assume \(\mathcal O\ne\emptyset\). The class \(\mathcal S_H\) is nonempty. Indeed, choose a constant
\[
    c\le \min_{\mathcal O} g
\]
and define
\[
    \underline u(x)
    =
    \begin{cases}
        c, & x\in\mathcal X\setminus\mathcal O,\\
        g(x), & x\in\mathcal O.
    \end{cases}
\]
Then every neighbor value of an interior vertex is at least \(c\), and hence
\[
    H_1[\underline u](x)\ge c=\underline u(x)
    \qquad\text{for }x\in\mathcal X\setminus\mathcal O.
\]
Thus \(\underline u\in\mathcal S_H\). The class is also bounded from above. If
\[
    C\coloneqq \max_{\mathcal O} g,
\]
then the constant function \(\overline u\equiv C\) satisfies
\[
    \overline u(x) \geq H_1[\overline u](x)
    \qquad x \in\mathcal X\setminus\mathcal O,
\]
and
\[
    \overline u (x)\ge g (x)
    \qquad x\in\mathcal O.
\]
By Theorem \ref{thm:homo-comparison-principle}, every \(v\in\mathcal S_H\) satisfies \(v\le \overline u\) in \(\mathcal X\).

\begin{theorem}\label{thm:homo-perron}
Assume \eqref{eq:H1-reachability} and \(\mathcal O\ne\emptyset\). Then the function \(u\colon\mathcal X\to\mathbb R\) defined by
\[
    u(x)\coloneqq \sup_{v\in\mathcal S_H} v(x)
\]
is the unique solution to \eqref{eq:H-M-A}.
\end{theorem}

\begin{proof}
Since \(\mathcal S_H\) is nonempty and bounded from above, \(u\) is well-defined. Moreover, because every \(v\in\mathcal S_H\) satisfies \(v=g\) on \(\mathcal O\), we immediately have
\[
    u=g
    \qquad\text{on }\mathcal O.
\]
We first show that \(u\) is a subsolution. Fix \(v\in\mathcal S_H\). Since \(v\le u\), the monotonicity of \(H_1\) gives
\[
    H_1[v](x)\le H_1[u](x)
    \qquad\text{for every }x\in\mathcal X\setminus\mathcal O.
\]
Therefore
\[
    v(x)\le H_1[v](x)\le H_1[u](x)
    \qquad\text{for }x\in\mathcal X\setminus\mathcal O.
\]
Taking the supremum over \(v\in\mathcal S_H\), we obtain
\[
    u(x)\le H_1[u](x)
    \qquad\text{for }x\in\mathcal X\setminus\mathcal O.
\]
Thus \(u\in\mathcal S_H\). It remains to prove the reverse inequality. Suppose, by contradiction, that there exists \(\tilde x\in\mathcal X\setminus\mathcal O\) such that
\[
    \delta
    \coloneqq
    H_1[u](\tilde x)-u(\tilde x)
    >0.
\]
Define \(\tilde u\colon\mathcal X\to\mathbb R\) by
\[
    \tilde u(x)
    =
    \begin{cases}
        u(x), & x\ne \tilde x,\\
        u(\tilde x)+\delta, & x=\tilde x.
    \end{cases}
\]
We claim that \(\tilde u\in\mathcal S_H\). Since \(\tilde x\notin\mathcal O\), the boundary values are unchanged:
\[
    \tilde u=g
    \qquad\text{on }\mathcal O.
\]

At \(x=\tilde x\), the operator \(H_1\) depends only on the values of the function on \(\mathcal N_{\tilde x}\), not on its value at \(\tilde x\). Hence
\[
    H_1[\tilde u](\tilde x)
    =
    H_1[u](\tilde x)
    =
    u(\tilde x)+\delta
    =
    \tilde u(\tilde x).
\]
If \(x\in\mathcal X\setminus\mathcal O\) and \(x\ne\tilde x\), then \(\tilde u(x)=u(x)\) and \(\tilde u\ge u\) in \(\mathcal X\). By monotonicity,
\[
    H_1[\tilde u](x)\ge H_1[u](x)\ge u(x)=\tilde u(x).
\]
Therefore \(\tilde u\in\mathcal S_H\). But this contradicts the definition of \(u\) as the pointwise supremum of all functions in \(\mathcal S_H\), since
\[
    \tilde u(\tilde x)
    =
    u(\tilde x)+\delta
    >
    u(\tilde x).
\]
Thus no such \(\tilde x\) exists, and
\[
    u(x)=H_1[u](x)
    \qquad\text{for every }x\in\mathcal X\setminus\mathcal O.
\]
Together with \(u=g\) on \(\mathcal O\), this proves that \(u\) solves \eqref{eq:H-M-A}. Uniqueness follows from Corollary \ref{cor:homo-uniqueness}.
\end{proof}

\section{Numerical schemes}\label{sct:numerical-schemes}

In this section, we give fixed-point schemes for the graph Monge--Amp\`ere equations. The schemes are motivated by the Bellman formulation and are fixed-point iterations for the corresponding Bellman maps. We write
\[
    \mathcal{U}\coloneqq \mathcal X\setminus\mathcal O.
\]
Throughout this section, we assume that every unlabeled vertex has exactly \(2d\) neighbors,
\[
    |\mathcal N_x|=2d,
    \qquad x\in \mathcal{U},
\]
and that the graph satisfies the reachability condition
\begin{equation}\label{eq:num-plucking-assumption}
    \forall\,\emptyset\ne S\subset \mathcal{U},
    \qquad
    \exists x\in S
    \quad\text{such that}\quad
    |\mathcal N_x\cap S|\le 1 .
\end{equation}
Equivalently, the subgraph induced by \(\mathcal{U}\) is a forest. 

For \(x\in \mathcal{U}\), order the neighbor values as
\[
    u(y_1)\le u(y_2)\le\cdots\le u(y_{2d}),
    \qquad y_i\in\mathcal N_x.
\]
We have
\[
    H_i[u](x) = \frac{u(y_{2i-1})+u(y_{2i})}{2},
    \qquad i=1,\dots,d,
\]
and, by Lemma \ref{lemma:eigen-expression}, we get
\[
    \lambda_i[u](x)=H_i[u](x)-u(x).
\]

\subsection{The fixed-point map}

Let \(f\colon \mathcal{U}\to(0,\infty)\).  For fixed values
\[
    h_1\le h_2\le\cdots\le h_d
\]
and \(F>0\), define \(T_F(h_1,\dots,h_d)\) to be the unique number \(t<h_1\) satisfying
\begin{equation}\label{eq:num-root-scalar}
    \prod_{i=1}^d (h_i-t)=F;
\end{equation}
recall the proof of Lemma \ref{lemma:restrict-set-alpha-2} for the existence of this number. The corresponding graph map is
\[
    T_f[u](x)
    \coloneqq
    T_{f(x)}\bigl(H_1[u](x),\dots,H_d[u](x)\bigr),
    \qquad x\in \mathcal{U},
\]
and we set 
$$T_f[u] (x)=g (x) , \qquad x\in \mathcal O .$$

\begin{lemma}\label{lem:num-root}
For every \(h_1\le \cdots\le h_d\) and \(F>0\), equation \eqref{eq:num-root-scalar} has a unique solution \(t<h_1\). Moreover, for \(d=2\),
\begin{equation}\label{eq:num-d2-root}
    T_F(h_1,h_2)
    =
    \frac{
        h_1+h_2
        -
        \sqrt{(h_2-h_1)^2+4F}
    }{2}.
\end{equation}
\end{lemma}

\begin{proof}
Consider
\[
    P(t)\coloneqq \prod_{i=1}^d(h_i-t)-F,
    \qquad t<h_1.
\]
On \((-\infty,h_1)\) each factor \(h_i-t\) is positive and
\[
    P'(t)
    =
    -\sum_{j=1}^d \prod_{\substack{i=1\\ i\ne j}}^d(h_i-t)
    <0.
\]
Thus \(P\) is strictly decreasing.  Moreover,
\[
    \lim_{t\to-\infty}P(t)=+\infty,
    \qquad
    P(h_1)=-F<0.
\]
Hence \(P\) has exactly one zero in \((-\infty,h_1)\). Formula \eqref{eq:num-d2-root} follows by solving the quadratic equation and selecting the root below \(h_1\).
\end{proof}

For \(d>2\), the value \(T_f[u](x)\) is computed robustly by bisection on the interval \((-\infty,H_1[u](x))\). More explicitly, set
\[
    P_x(t)=\prod_{i=1}^d\bigl(H_i[u](x)-t\bigr)-f(x).
\]
Start with \(b=H_1[u](x)\) and \(a=b-\eta\), where \(\eta>0\). Double \(\eta\) until \(P_x(a)>0\).  Since \(P_x(b)<0\), bisection on \([a,b]\) converges to the desired root. This avoids computing all polynomial roots and therefore avoids selecting a wrong branch.

The inhomogeneous numerical scheme is
\begin{equation}\label{eq:num-inhom-scheme}
    \begin{cases}
    u^{m+1}(x)
    =
    (1-\omega)u^m(x)+\omega T_f[u^m](x),
    &x\in \mathcal{U},\\[0.7ex]
    u^{m+1}(x)=g(x),
    &x\in\mathcal O,
    \end{cases}
\end{equation}
where
\[
    0<\omega\le 1.
\]
For \(d=2\), this becomes the explicit update
\begin{equation}\label{eq:num-inhom-d2}
\begin{aligned}
    u^{m+1}(x)
    &=
    (1-\omega)u^m(x)
    +
    \frac{\omega}{2}
    \Bigl(
        H_1[u^m](x)+H_2[u^m](x) \\
    &\qquad
        -
        \sqrt{
            \bigl(H_2[u^m](x)-H_1[u^m](x)\bigr)^2
            +4f(x)
        }
    \Bigr).
\end{aligned}
\end{equation}

\begin{remark}
If \(u=T_f[u]\) in \(\mathcal{U}\), then \(u(x)<H_1[u](x)\) and
\[
    \prod_{i=1}^d\bigl(H_i[u](x)-u(x)\bigr)=f(x),
    \qquad x\in \mathcal{U}.
\]
Thus, every fixed point of \(T_f\) is a strictly graph-convex solution of the inhomogeneous graph Monge--Amp\`ere equation.  Conversely, every strictly graph-convex solution of the equation is a fixed point of \(T_f\).
\end{remark}

For the homogeneous equation, graph convexity gives
\[
    \mathcal M [u](x)=0
    \quad\Longleftrightarrow\quad
    H_1[u](x)-u(x)=0.
\]
Define
\[
    T_0[u](x)\coloneqq H_1[u](x),
    \qquad x\in \mathcal{U},
\]
and set \(T_0[u]=g\) on \(\mathcal O\). The homogeneous scheme is
\begin{equation}\label{eq:num-hom-scheme}
    \begin{cases}
    u^{m+1}(x)
    =
    (1-\omega)u^m(x)+\omega H_1[u^m](x),
    &x\in \mathcal{U},\\[0.7ex]
    u^{m+1}(x)=g(x),
    &x\in\mathcal O,
    \end{cases}
\end{equation}
again with \(0<\omega\le 1\).

\subsection{Convergence on woven forests}

The convergence proof uses the barrier supplied by Theorem \ref{thm:plucking-condition}.  Under \eqref{eq:num-plucking-assumption}, the sufficiency construction in Theorem \ref{thm:plucking-condition} gives a function \(b\colon\mathcal X\to[0,\infty)\) such that
\[
    b=0
    \quad\text{on }\mathcal O,
    \qquad
    b(x)>0
    \quad\text{for }x\in \mathcal{U},
\]
and
\begin{equation}\label{eq:num-barrier}
    \max_{\substack{y,z\in\mathcal N_x\\ y\ne z}}
    \frac{b(y)+b(z)}2
    \le b(x)-1,
    \qquad x\in \mathcal{U}.
\end{equation}
Since $b\ge0$ on $\mathcal X$ and every $x\in \mathcal{U}$ has at least two
neighbors, the left-hand side of \eqref{eq:num-barrier} is
nonnegative. Hence $b(x)\ge1$ for every $x\in \mathcal{U}$. We let
\[
    B\coloneqq \max_{x\in \mathcal{U}} b(x),
\]
and, for functions \(w\) satisfying \(w=0\) on \(\mathcal O\), define
\[
    \|w\|_b
    \coloneqq
    \max_{x\in \mathcal{U}}\frac{|w(x)|}{b(x)}.
\]
This is well-defined, since \(\mathcal{U}\) is finite. Moreover, this is a norm on the space of functions vanishing on \(\mathcal O\).

\begin{lemma}\label{lem:num-H-contraction}
Let \(u,v\colon\mathcal X\to\mathbb R\) satisfy \(u=v=g\) on \(\mathcal O\).
If
\[
    \mu\coloneqq \|u-v\|_b,
\]
then, for every \(x\in \mathcal{U}\) and \(i=1,\dots,d\),
\begin{equation}\label{eq:num-H-estimate}
    |H_i[u](x)-H_i[v](x)|
    \le
    \mu\bigl(b(x)-1\bigr).
\end{equation}
\end{lemma}

\begin{proof}
Since \(\mu=\|u-v\|_b\), we have
\[
    u\le v+\mu b
    \quad\text{and}\quad
    v\le u+\mu b
    \qquad\text{in }\mathcal X.
\]
Fix \(x\in \mathcal{U}\) and \(i\in\{1,\dots,d\}\). Let \(S_i\subset\mathcal N_x\), \(|S_i|=2i\), be a minimizing set for \(H_i[v](x)\). Then
\begin{align*}
    H_i[u](x)
    &\le
    \max_{\substack{y,z\in S_i\\ y\ne z}}
    \frac{u(y)+u(z)}2                                    \\
    &\le
    \max_{\substack{y,z\in S_i\\ y\ne z}}
    \left[
        \frac{v(y)+v(z)}2
        +
        \mu\frac{b(y)+b(z)}2
    \right]                                               \\
    &\le
    H_i[v](x)+\mu\bigl(b(x)-1\bigr),
\end{align*}
where we used \eqref{eq:num-barrier} in the last step. Reversing the roles of \(u\) and \(v\) gives the reverse inequality, and hence \eqref{eq:num-H-estimate} follows.
\end{proof}

\begin{lemma}\label{lem:num-T-monotone}
For fixed \(F>0\), the map
\[
    (h_1,\dots,h_d)\mapsto T_F(h_1,\dots,h_d)
\]
is monotone nondecreasing in each variable and has the following property:
\[
    T_F(h_1+c,\dots,h_d+c)=T_F(h_1,\dots,h_d)+c.
\]

The same properties hold for \(T_0(h_1,\dots,h_d)=h_1\).
\end{lemma}

\begin{proof}
Using the Bellman representation, Theorem \ref{theorem:equivalent-formulation}, we have
\[
    T_F(h_1,\dots,h_d)
    =
    \inf_{\alpha\in\mathcal A_d}
    \frac{
        \sum_{i=1}^d \alpha_i h_i-dF^{1/d}
    }{
        \sum_{i=1}^d\alpha_i
    }.
\]
Monotonicity in the variables \(h_i\) is immediate from this formula. Also, adding a constant \(c\) to each \(h_i\) adds
\[
    \frac{c\sum_i\alpha_i}{\sum_i\alpha_i}=c
\]
inside the infimum, proving \[
T_F(h_1+c,\dots,h_d+c)=T_F(h_1,\dots,h_d)+c.
\] 

The case \(T_0=h_1\) is analogous.
\end{proof}

Now we prove the convergence of the inhomogeneous scheme.

\begin{theorem}\label{thm:num-inhom-convergence}
Assume \eqref{eq:num-plucking-assumption}, \(|\mathcal N_x|=2d\) for every \(x\in \mathcal{U}\), and \(f>0\) in \(\mathcal{U}\).  Then the map \(T_f\) satisfies
\[
    \|T_f[u]-T_f[v]\|_b
    \le
    \left(1-\frac1B\right)\|u-v\|_b
\]
for all \(u,v\) with \(u=v=g\) on \(\mathcal O\). Consequently, for every initial datum $u^0$ satisfying $u^0=g$ on $\mathcal O$, the iteration \eqref{eq:num-inhom-scheme} converges to the unique fixed point $u^\ast$ of $T_f$, equivalently to the unique solution of
\[
    \begin{cases}
    \displaystyle
    \prod_{i=1}^d\bigl(H_i[u](x)-u(x)\bigr)=f(x),
    &x\in \mathcal{U},\\[1ex]
    u(x)=g(x),
    &x\in\mathcal O,
    \end{cases}
\]
that is strictly graph convex on $\mathcal{U}$. Moreover,
\begin{equation}\label{eq:num-inhom-rate}
    \|u^m-u^\ast\|_b
    \le
    \left(1-\frac{\omega}{B}\right)^m
    \|u^0-u^\ast\|_b .
\end{equation}
\end{theorem}

\begin{proof}
Let \(\mu=\|u-v\|_b\). By Lemma \ref{lem:num-H-contraction},
\[
    H_i[u](x)
    \le
    H_i[v](x)+\mu(b(x)-1),
    \qquad i=1,\dots,d.
\]
Now, from Lemma \ref{lem:num-T-monotone},
we get
\[
    T_f[u](x)
    \le
    T_f[v](x)+\mu(b(x)-1).
\]
Reversing the roles of \(u\) and \(v\), we obtain
\[
    |T_f[u](x)-T_f[v](x)|
    \le
    \mu(b(x)-1).
\]
Dividing by \(b(x)\) and taking the maximum over \(\mathcal{U}\) gives
\[
    \|T_f[u]-T_f[v]\|_b
    \le
    \left(1-\frac1B\right)\|u-v\|_b.
\]
Now define the damped map
\[
    S_{\omega,f}[u]
    \coloneqq
    (1-\omega)u+\omega T_f[u]
    \quad\text{in }\mathcal{U},
    \qquad
    S_{\omega,f}[u]=g
    \quad\text{on }\mathcal O.
\]
The preceding estimate gives
\[
    \|S_{\omega,f}[u]-S_{\omega,f}[v]\|_b
    \le
    \left(1-\frac{\omega}{B}\right)\|u-v\|_b.
\]
Since \(0<\omega\le1\), this is a strict contraction. Banach's fixed-point theorem yields a unique fixed point \(u^\ast\) and the rate \eqref{eq:num-inhom-rate}. By the definition of \(T_f\), this fixed point is exactly the strictly graph-convex solution of the inhomogeneous graph Monge--Amp\`ere equation.
\end{proof}

Next, we turn our attention to the proof of the convergence of the homogeneous scheme.

\begin{theorem}\label{thm:num-hom-convergence}
Assume \eqref{eq:num-plucking-assumption} and \(|\mathcal N_x|=2d\) for every \(x\in \mathcal{U}\). Then
\[
    \|T_0[u]-T_0[v]\|_b
    \le
    \left(1-\frac1B\right)\|u-v\|_b
\]
for all \(u,v\) with \(u=v=g\) on \(\mathcal O\). Consequently, for every initial datum \(u^0=g\) on \(\mathcal O\), the iteration \eqref{eq:num-hom-scheme} converges to the unique solution \(u^\ast\) of
\[
    \begin{cases}
    u(x)=H_1[u](x),
    &x\in \mathcal{U},\\
    u(x)=g(x),
    &x\in\mathcal O.
    \end{cases}
\]
Moreover,
\[
    \|u^m-u^\ast\|_b
    \le
    \left(1-\frac{\omega}{B}\right)^m
    \|u^0-u^\ast\|_b .
\]
\end{theorem}

\begin{proof}
The proof follows by the same arguments used for the proof of Theorem \ref{thm:num-inhom-convergence}, using \(T_0[u]=H_1[u]\) and Lemma \ref{lem:num-H-contraction}.
\end{proof}

\subsection{Numerical experiments}

For the numerical tests, we consider an analog of the two-dimensional continuous Monge--Amp\`ere problem in the ball $B_1(0)$. The continuous problem is stated as follows.
\[
    \begin{cases}
        \det (D^2 u)(x) =1, & x\in B_1(0),\\
        u(x)=\frac{1}{2}|x|^2=0.5, & x\in \partial B_1(0).
    \end{cases}
\]

The unique convex solution to this problem is given by $$u(x) = \frac{1}{2}|x|^2.$$

The aforementioned discrete analog of this problem is
\begin{equation}\label{eq:numerical-test-equations}
    \begin{cases}
        \mathcal{M}[u](x) =1, & x\in \mathcal{X}\setminus\mathcal{O},\\
        u(x)=0.5, & x\in \mathcal{O},
    \end{cases}
\end{equation}
where $\mathcal{X}\setminus\mathcal{O}$ and $\mathcal{O}$ are finite subsets of $B_1(0)$ and $\partial B_1(0)$, respectively.
 
So, we aim to obtain a numerical solution given by \eqref{eq:numerical-test-equations} approximating $u(x) = \frac{1}{2}|x|^2$.

\subsubsection{Setup}

Fix integers $1\leq M\leq N$. We take as labeled vertices the $M$ equally spaced points on the unit circle,
\[
\mathcal{O}
\coloneqq
\left\{
x_j=(\cos\theta_j,\sin\theta_j)\in\partial B_1(0)
\,:\,
\theta_j=2\pi\frac{j-1}{M},
\quad j=1,\ldots,M
\right\}.
\]
Thus, $|\mathcal{O}|=M$. For each of the graph structures shown in Figure \ref{fig:woven-forests}, we construct a forest using a chosen unlabeled set $\mathcal U\subset B_1(0)\setminus\mathcal O$, with $|\mathcal{U}|=N-M$, so that the vertex set of the $\mathcal{O}$-woven forest $\mathcal{X}\coloneqq\mathcal{U}\cup\mathcal{O}$ has $|\mathcal{X}|=|\mathcal{U}|+|\mathcal{O}|=N$ vertices in total.

\begin{figure}[ht]
\includegraphics[width=0.35\linewidth]{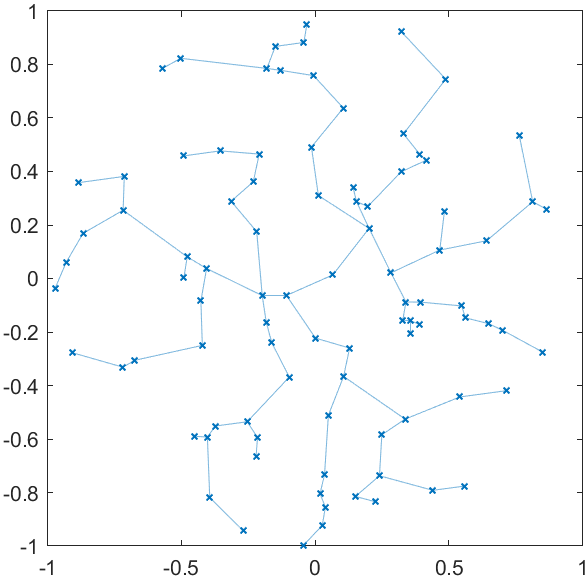}\;
    \includegraphics[width=0.35\linewidth]{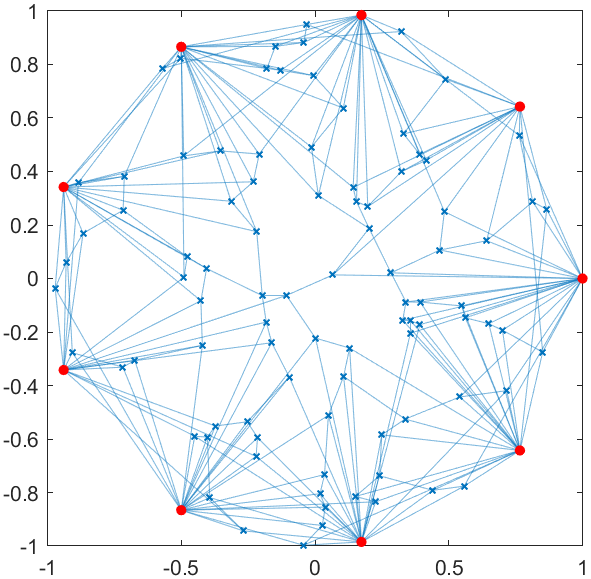}\\
    \medskip
    \includegraphics[width=0.35\linewidth]{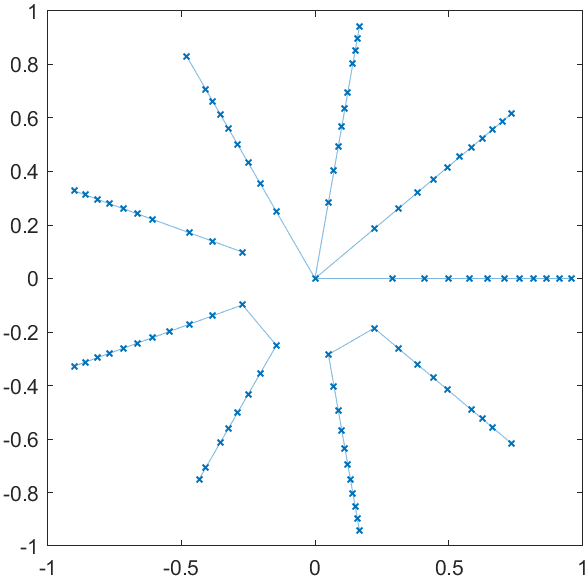}\;
    \includegraphics[width=0.35\linewidth]{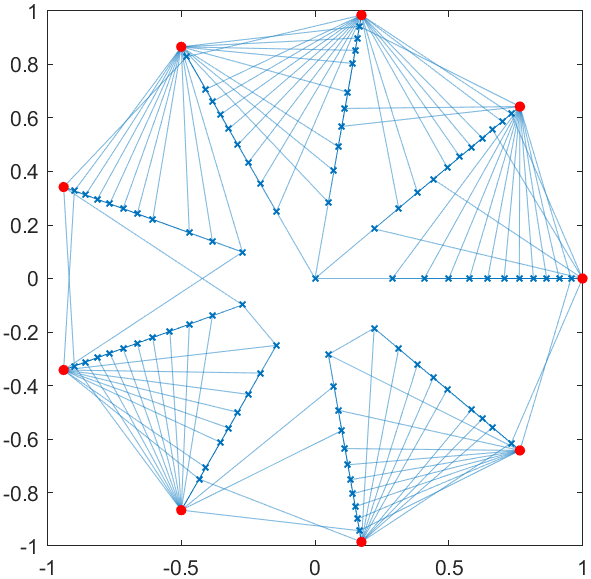}\\
    \medskip
    \includegraphics[width=0.35\linewidth]{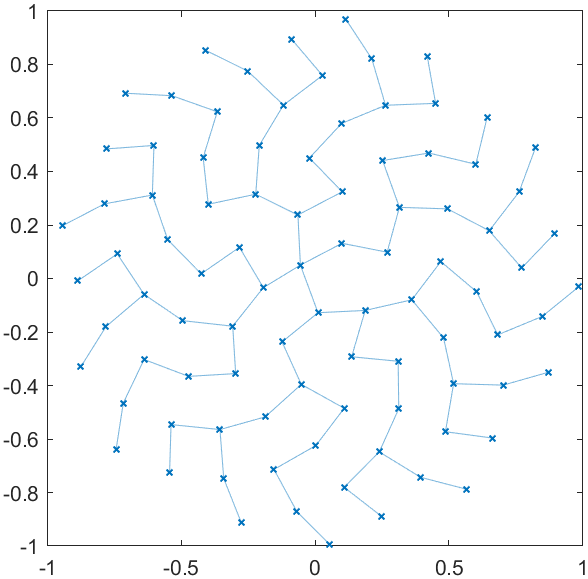}\;
    \includegraphics[width=0.35\linewidth]{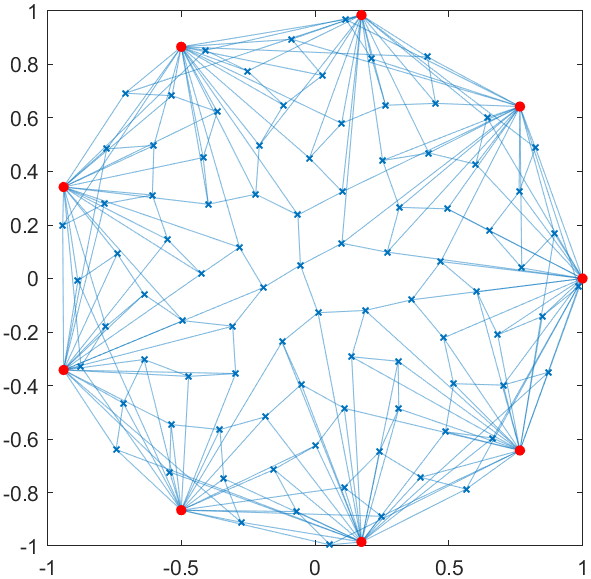}\\
    \medskip
    \includegraphics[width=0.35\linewidth]{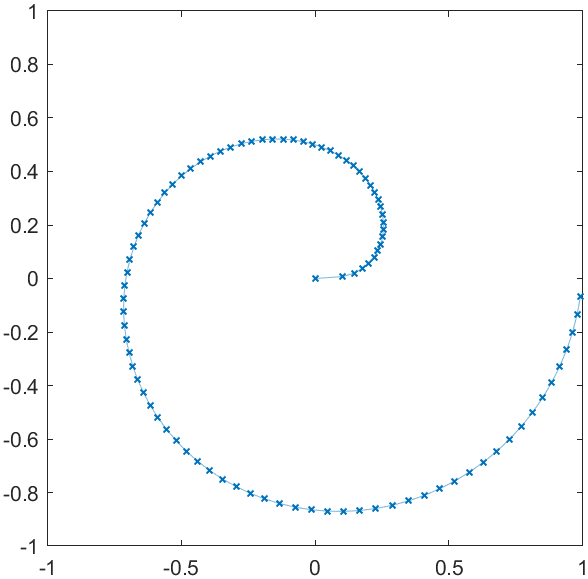}\;
    \includegraphics[width=0.35\linewidth]{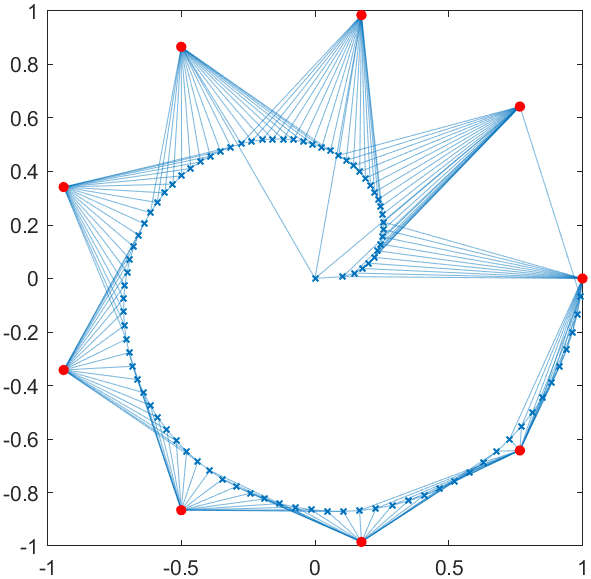}
    \caption{Left column: Forests. Right column: Corresponding woven forests. Top to bottom: randomly sampled, rays shaped, uniformly distanced, spiral shaped.}
    \label{fig:woven-forests}
\end{figure}

\clearpage 

\subsubsection{Inhomogeneous Monge--Amp\`ere}
Now we present results for the inhomogeneous problem. We will perform various tests by varying the graph structure and examining the connectivity of vertices that yield a solution reproducing the parabolic shape of the continuous solution. We include a low-opacity parabola in each for comparison with the solution of the continuous analog of the problem, as shown in Figure \ref{fig:parabola}. We also normalize the results for easier visualization using the scale
$$u_{\text{scaled}} = (u-0.5)\left(\frac{0.5}{0.5-\min(u)}\right)+0.5.$$
This ensures that the results are in the same interval $[0, 0.5]$ as the parabola while preserving their proportions.

We fix $u^0 \equiv 0.5$, $N=100$, $M=9$ and a maximum of $10^4$ iterations. We also fix the step size to be $\omega=1$ for all tests in the iterative scheme \eqref{eq:num-inhom-scheme}.

\begin{figure}[ht]
    \centering
    \includegraphics[width=0.6\linewidth]{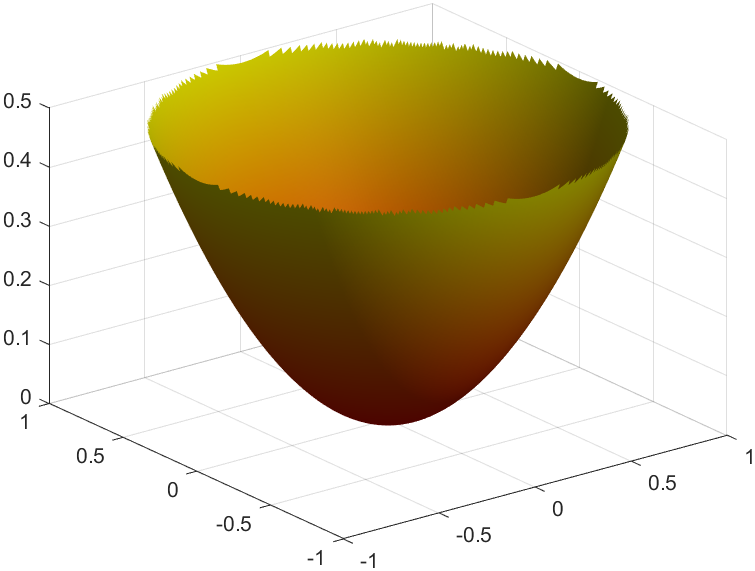}
    \caption{Plot of $\frac{1}{2}|x|^2$.}
    \label{fig:parabola}
\end{figure}

\begin{figure}[ht]
    \centering
    \includegraphics[width=0.45\linewidth]{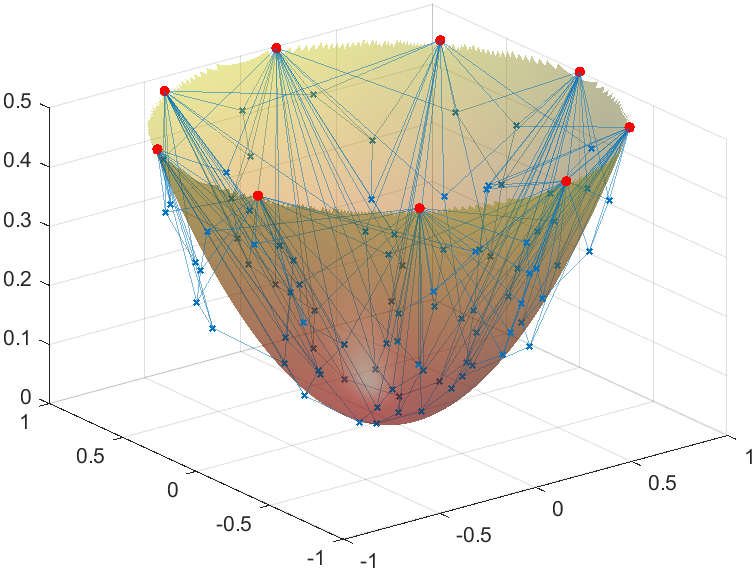}\;
    \includegraphics[width=0.45\linewidth]{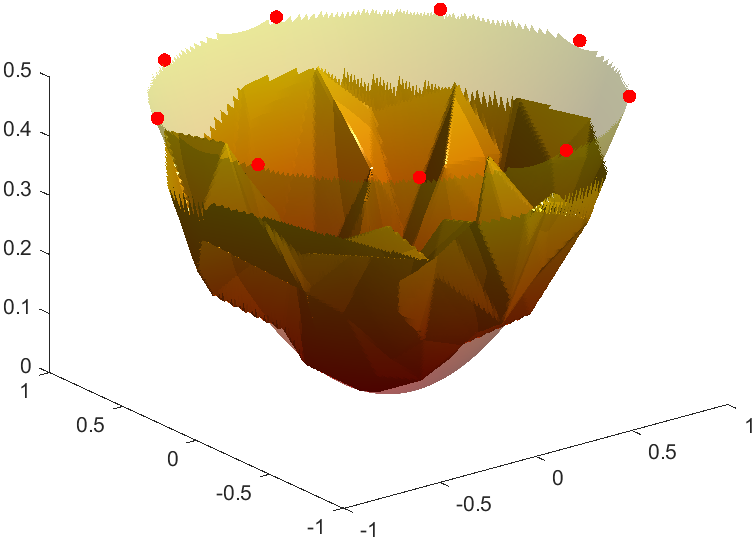}\\
    \includegraphics[width=0.45\linewidth]{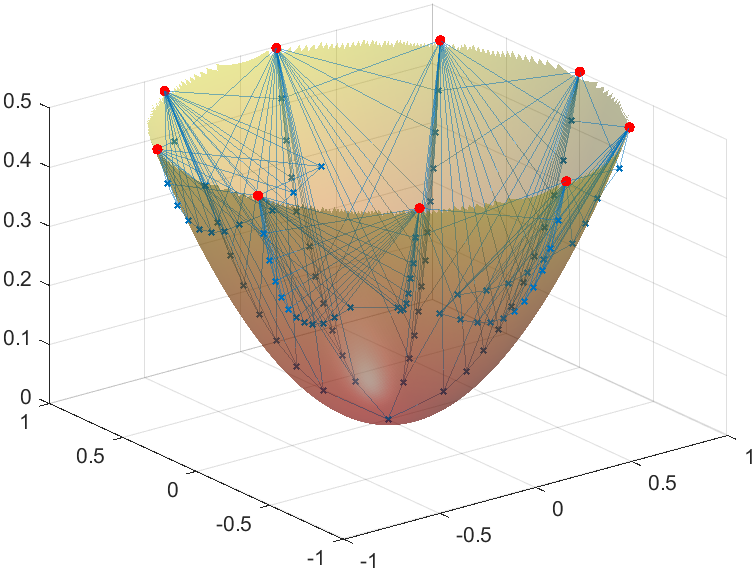}\;
    \includegraphics[width=0.45\linewidth]{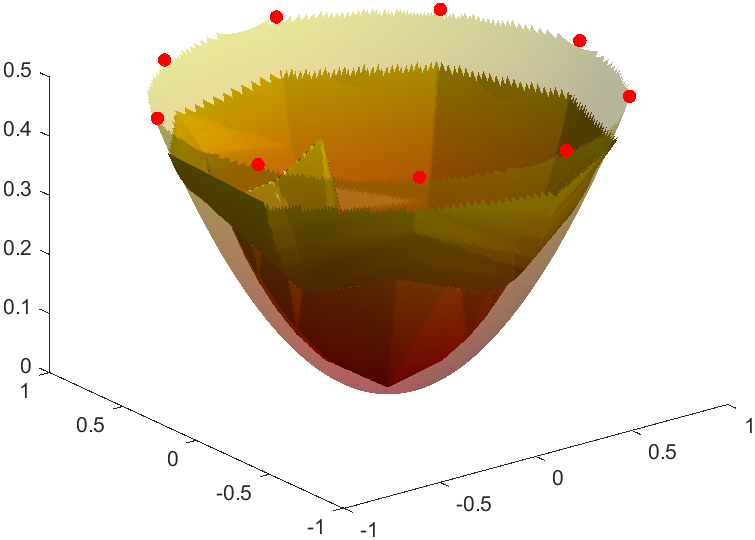}\\
    \includegraphics[width=0.45\linewidth]{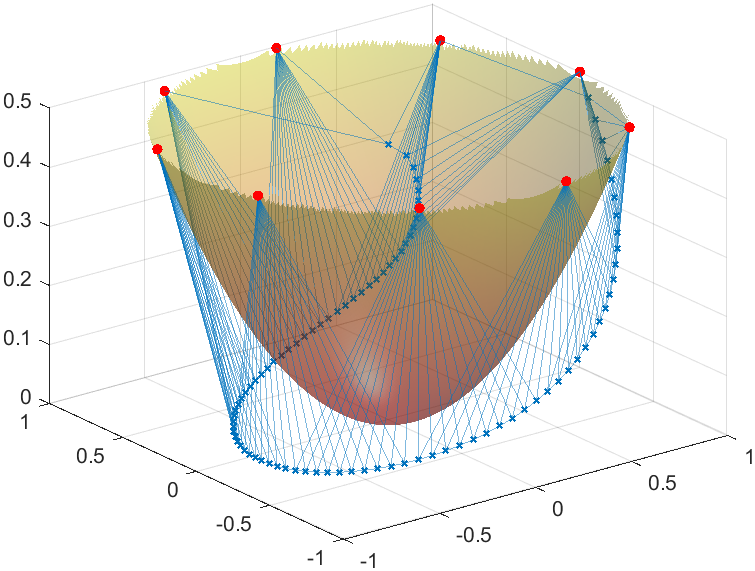}\;
    \includegraphics[width=0.45\linewidth]{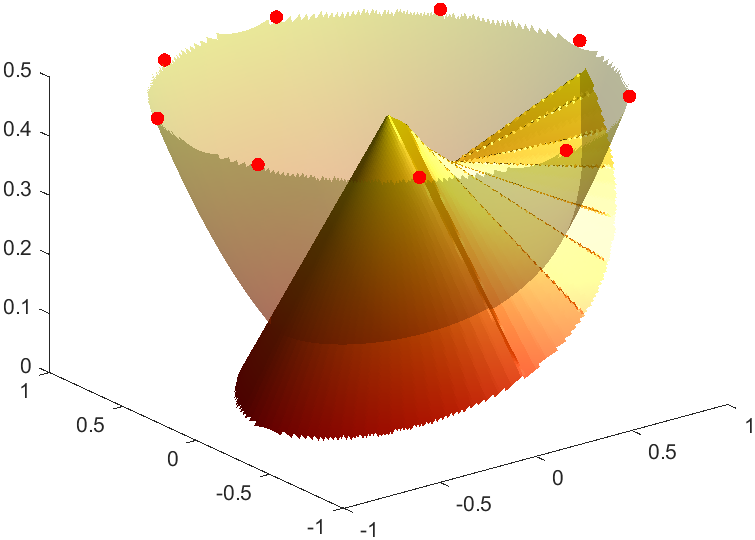}\\
    \includegraphics[width=0.45\linewidth]{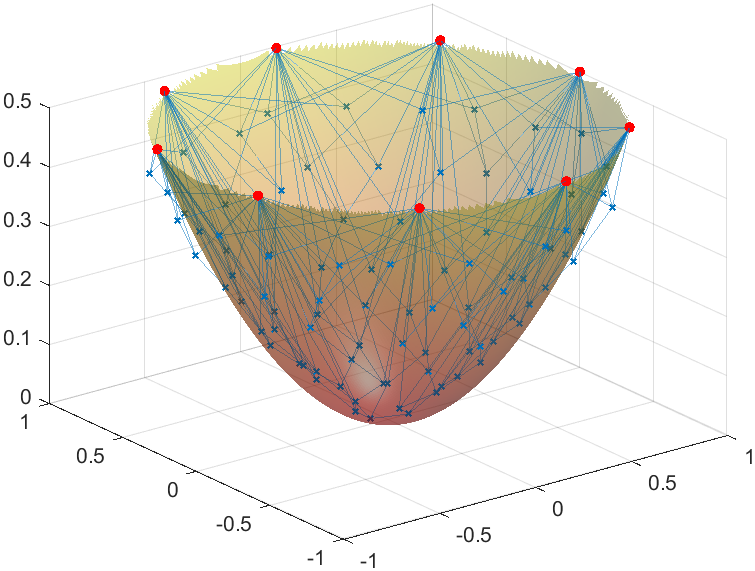}\;
    \includegraphics[width=0.45\linewidth]{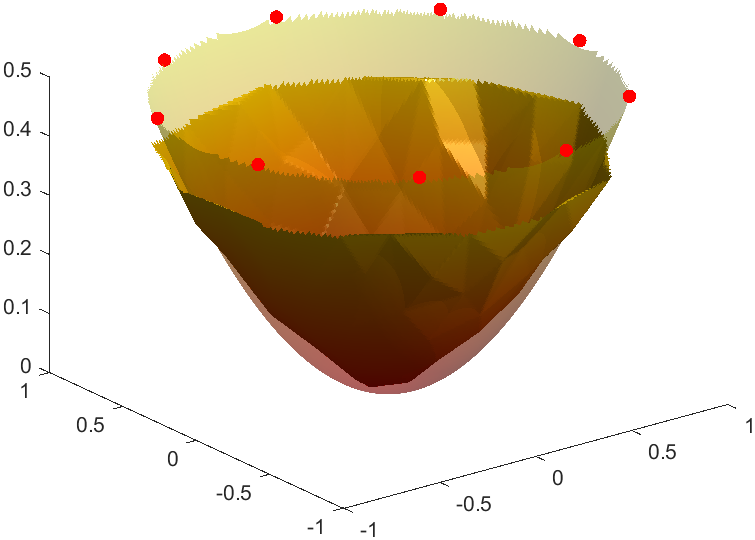}
    
    \caption{Solutions to the inhomogeneous problem. Same order as in Figure \ref{fig:woven-forests}. For better clarity, surface plots (right) exclude the labeled vertices.}
    \label{fig:woven-forests-result}
\end{figure}

We can see in Figure \ref{fig:woven-forests-result} that most results take on a shape similar to a parabola, except the spiral graph. This could be because the geometry of the spiral graph differs markedly from that of the unit ball, whereas the other graphs better emulate that geometry. Moreover, except for the spiral, each graph branches outward from a vertex located near the center of the ball, giving rise to natural levels determined by the graph distance from this central vertex. This layered structure is the main factor shaping the appearance of the solution plot. For example, the random graph has many `spikes', which correspond to high-level vertices close to the origin. The solution for the rays graph contains visual `discontinuities'. The reason could be that the interior forest is not a single tree, so while the graph is fully connected, the scheme does not propagate through elements of $\mathcal{O}$, making each connected component effectively its own graph. The most parabolically shaped results we have found are for the radial tree in Figure \ref{fig:mold-forest}.

\begin{figure}[ht]
    \centering
    \includegraphics[width=0.4\linewidth]{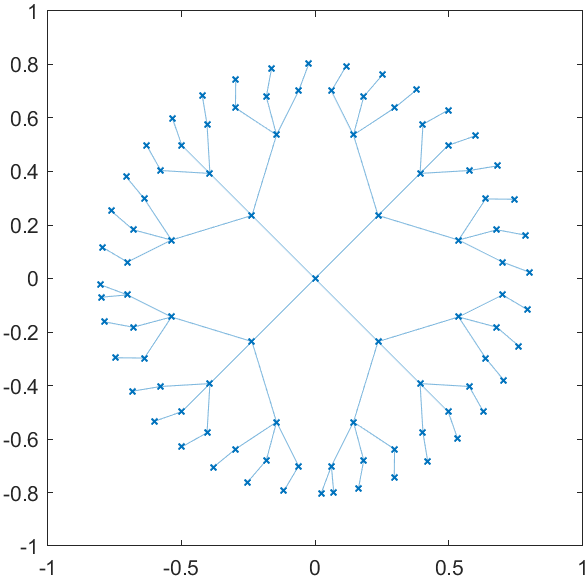}\qquad
    \includegraphics[width=0.4\linewidth]{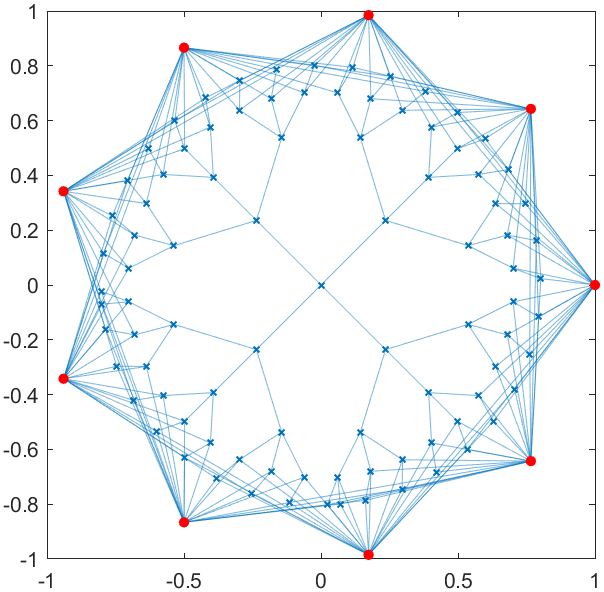}\\
    \medskip
    \includegraphics[width=0.45\linewidth]{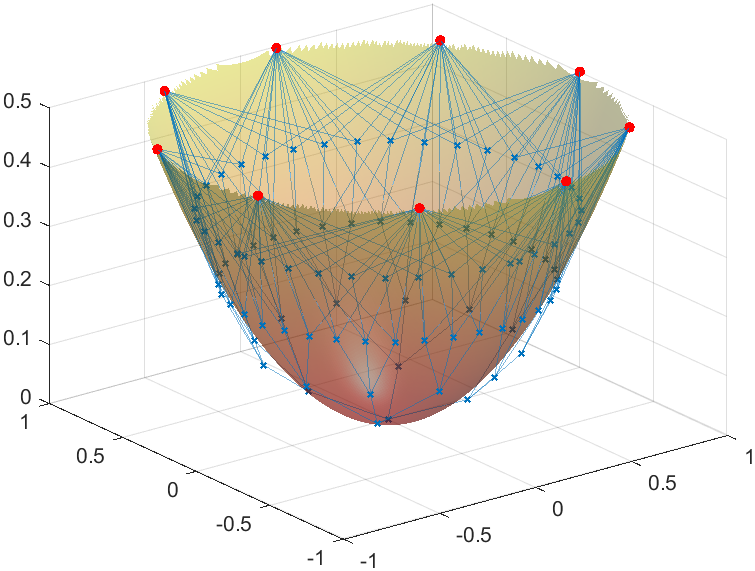}\;
    \includegraphics[width=0.45\linewidth]{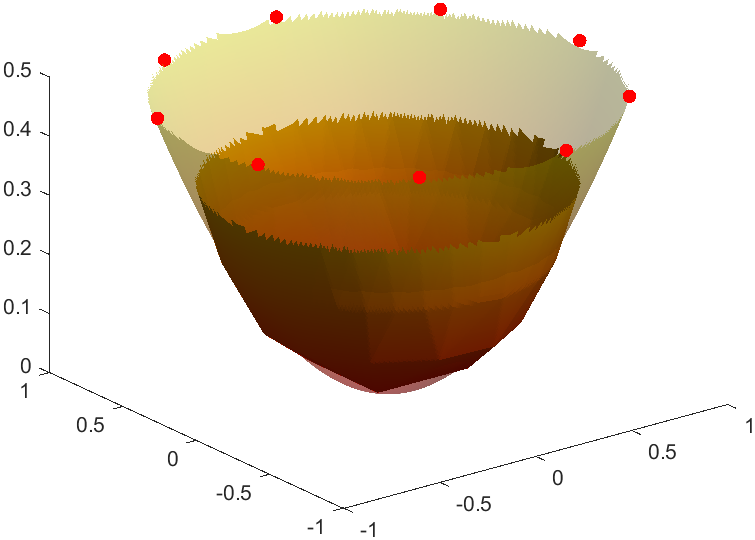}
    \caption{Results for the radial tree.}
    \label{fig:mold-forest}
\end{figure}

\subsubsection{Comparison with graph Laplacian}

In the continuum setting, the Diri\-chlet problem 
\begin{equation}
    \begin{cases}
        \Delta u(x) = 2, & x\in B_1(0),\\
        u(x)=\frac{1}{2}|x|^2=0.5, & x\in \partial B_1(0)
    \end{cases}
\end{equation}
has the same solution as the Monge--Amp\`ere equation with $f=1$. We will be using this fact as a point of comparison between graph-Laplacian regularization and our graph Monge--Amp\`ere regularization, pitting \eqref{eq:numerical-test-equations} against the graph Dirichlet problem
\begin{equation}\label{eq:dirichlet-poisson-equation}
    \begin{cases}
        \mathcal{L}[u](x) =2, & x\in \mathcal{X}\setminus\mathcal{O},\\
        u(x)=0.5, & x\in \mathcal{O}.
    \end{cases}
\end{equation}
Numerical solutions of \eqref{eq:dirichlet-poisson-equation} were
computed using the methods of Flores, Calder, and Lerman
\cite{FloresCalderLerman2022}, followed by the same rescaling as
above.

\begin{figure}[ht]
    \centering
    \includegraphics[width=0.45\linewidth]{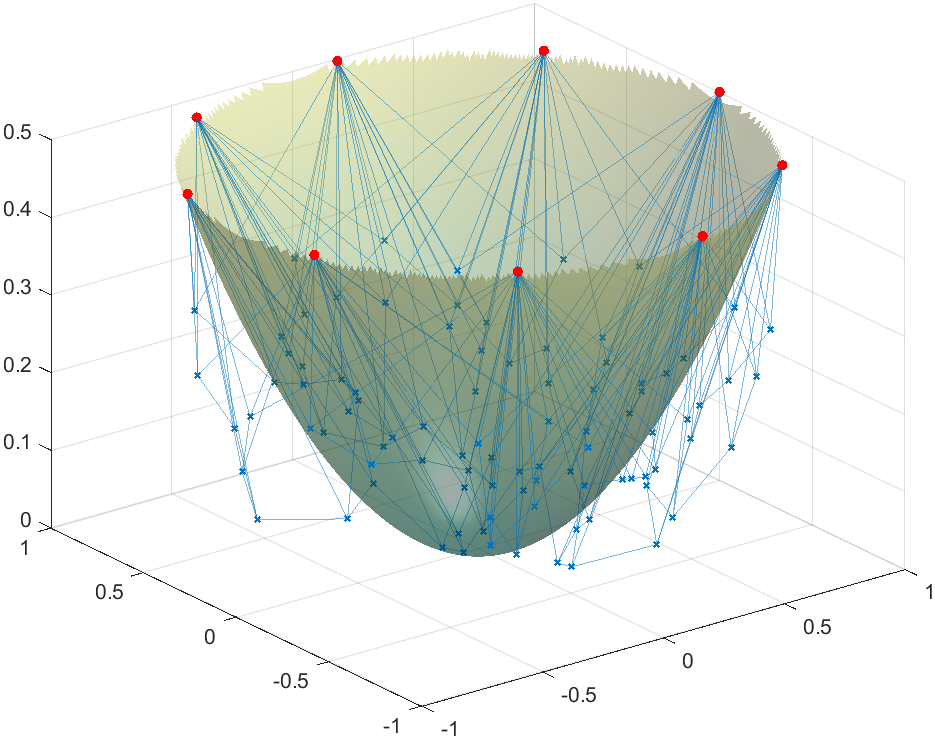}\;
    \includegraphics[width=0.45\linewidth]{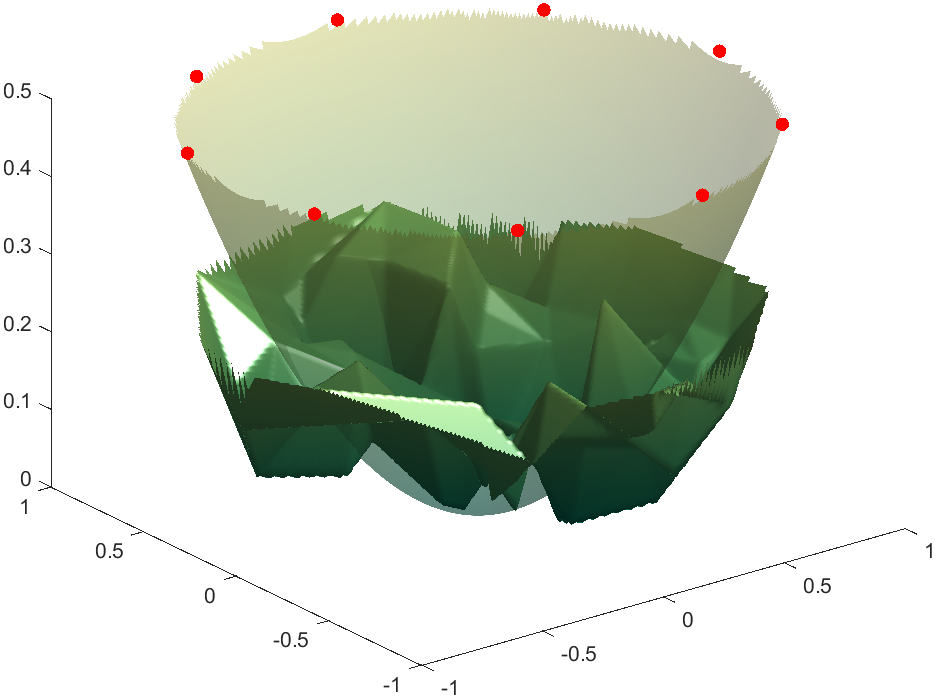}\\
    \includegraphics[width=0.45\linewidth]{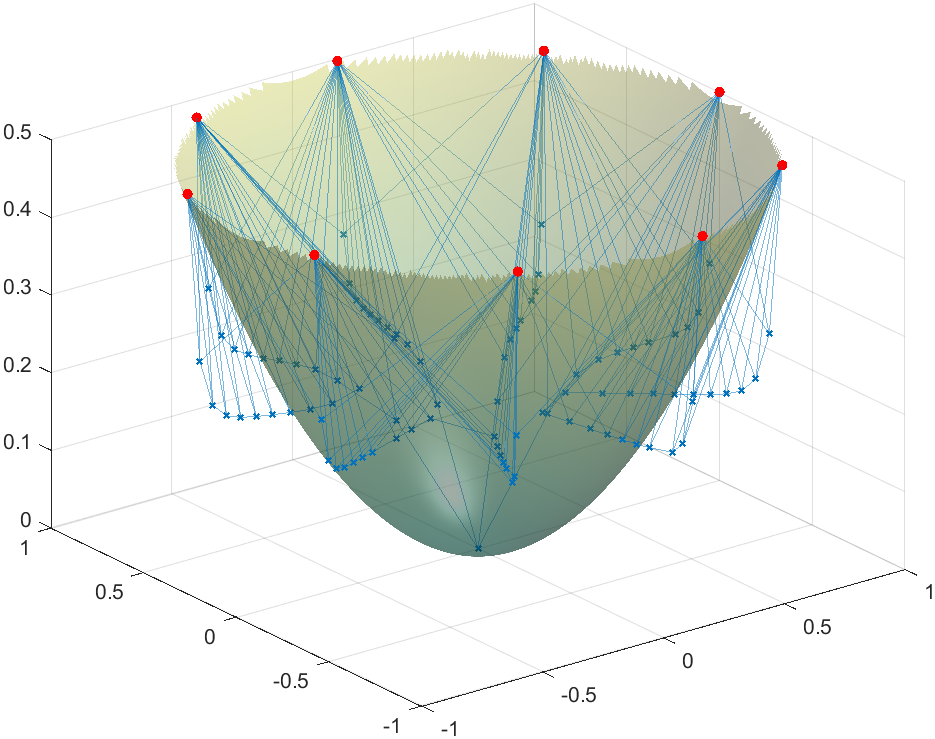}\;
    \includegraphics[width=0.45\linewidth]{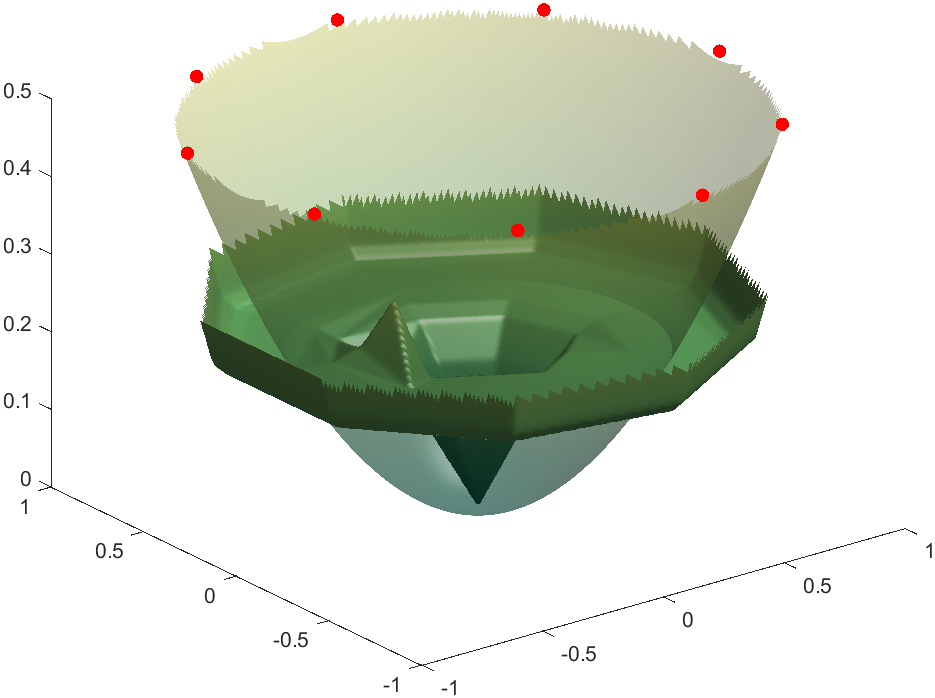}\\
    \includegraphics[width=0.45\linewidth]{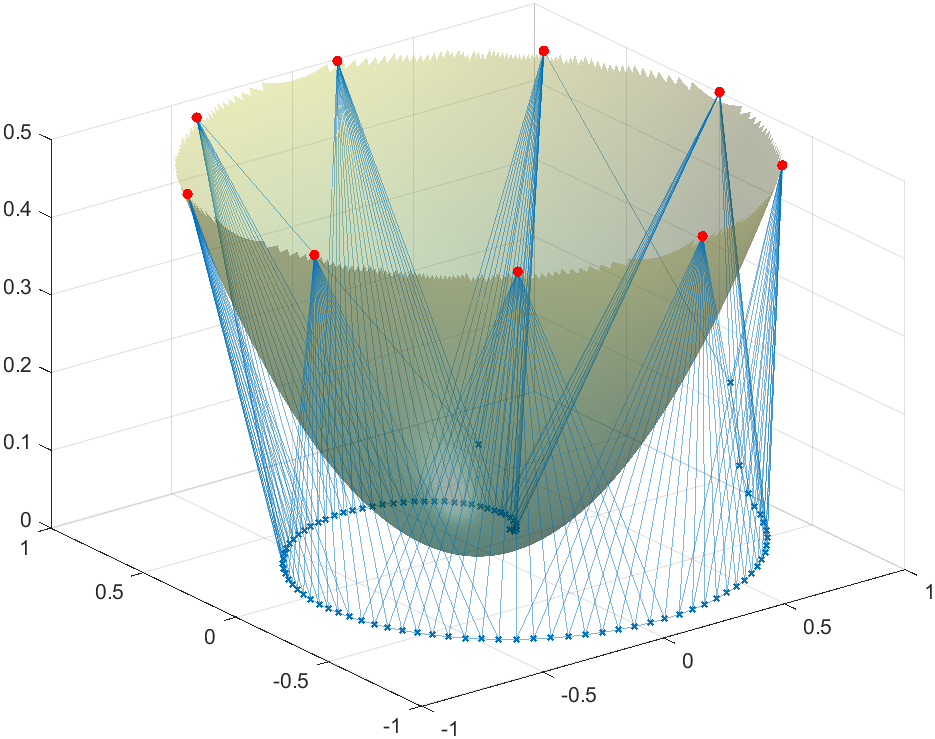}\;
    \includegraphics[width=0.45\linewidth]{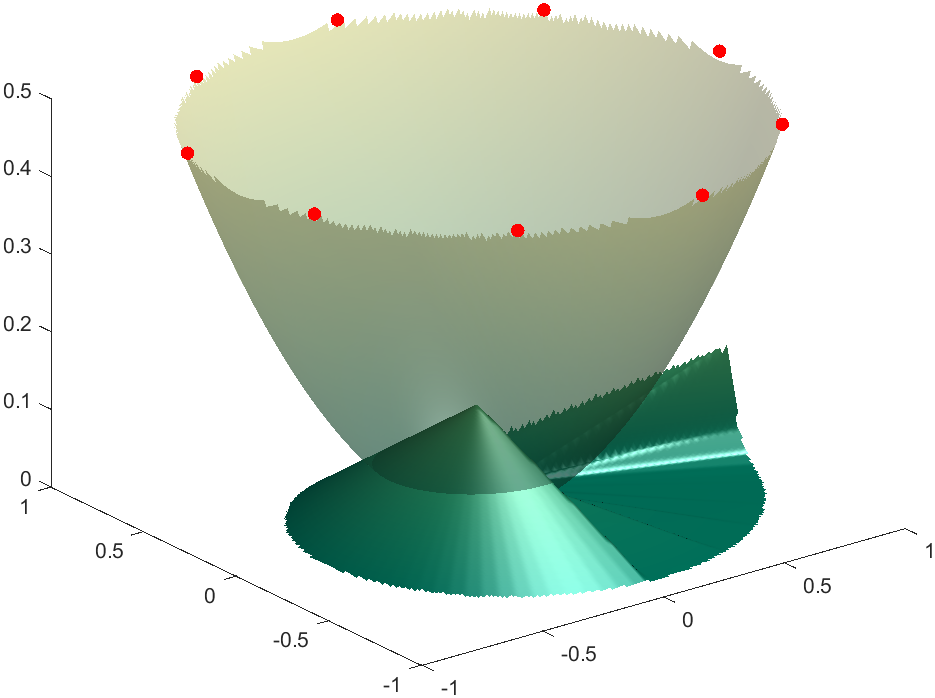}\\
    \includegraphics[width=0.45\linewidth]{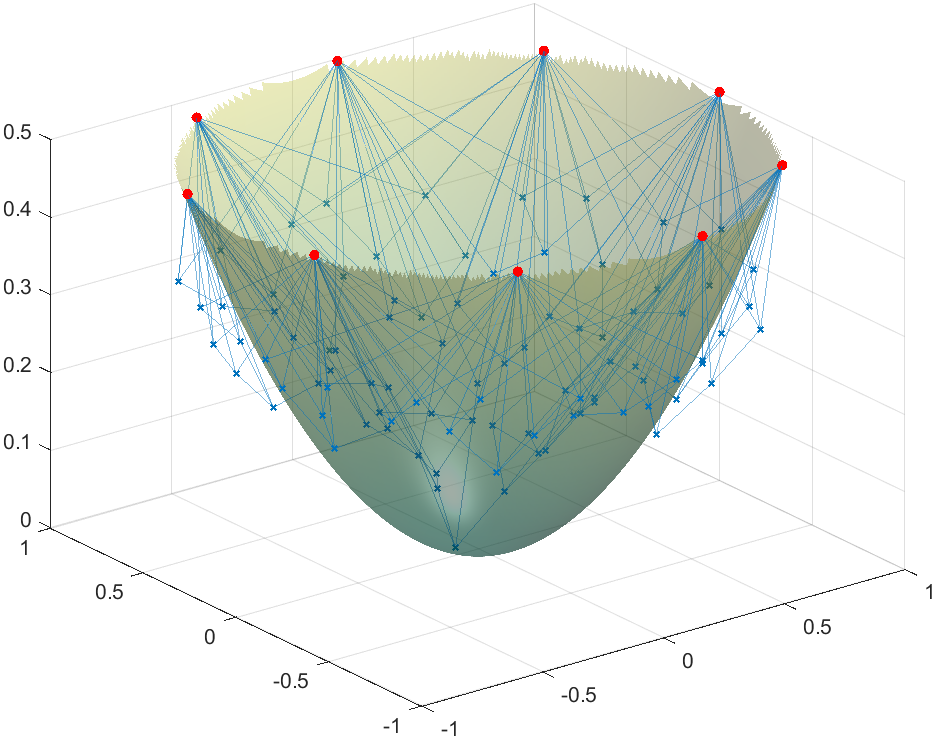}\;
    \includegraphics[width=0.45\linewidth]{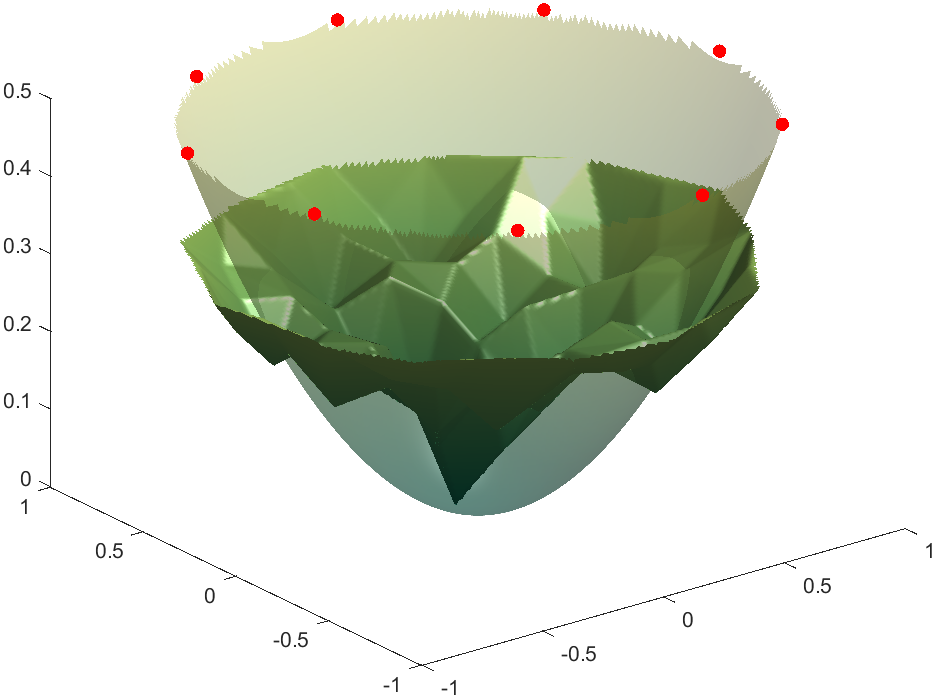}
    \caption{Solution of \eqref{eq:dirichlet-poisson-equation} on the graphs in Figure \ref{fig:woven-forests}.}
    \label{fig:laplace-regularization}
\end{figure}

\begin{figure}[ht]
    \centering
    \includegraphics[width=0.45\linewidth]{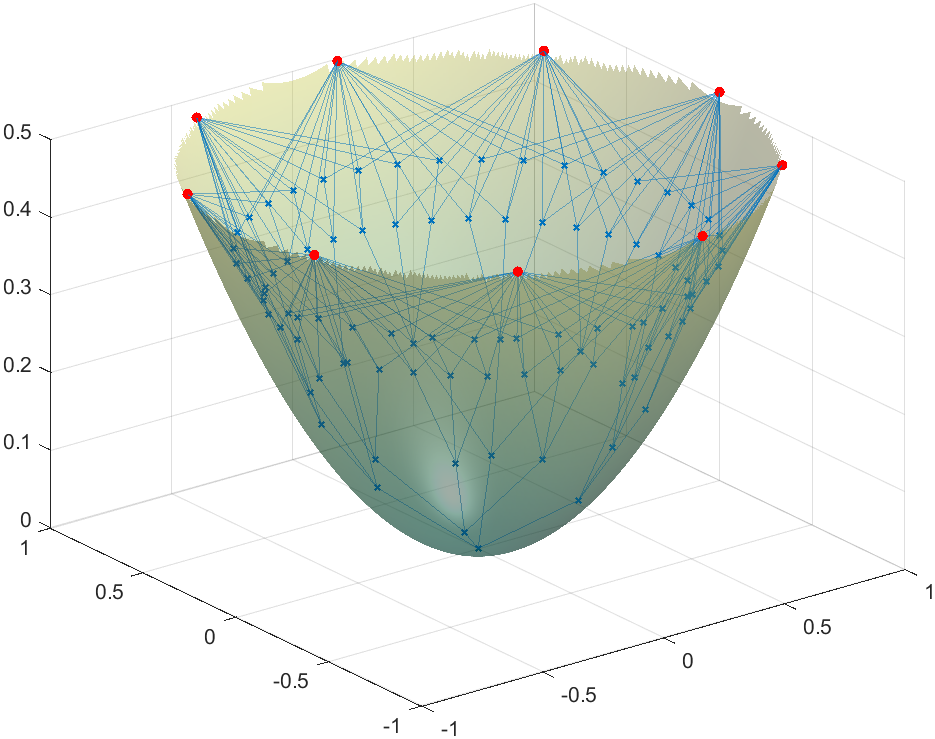}\;
    \includegraphics[width=0.45\linewidth]{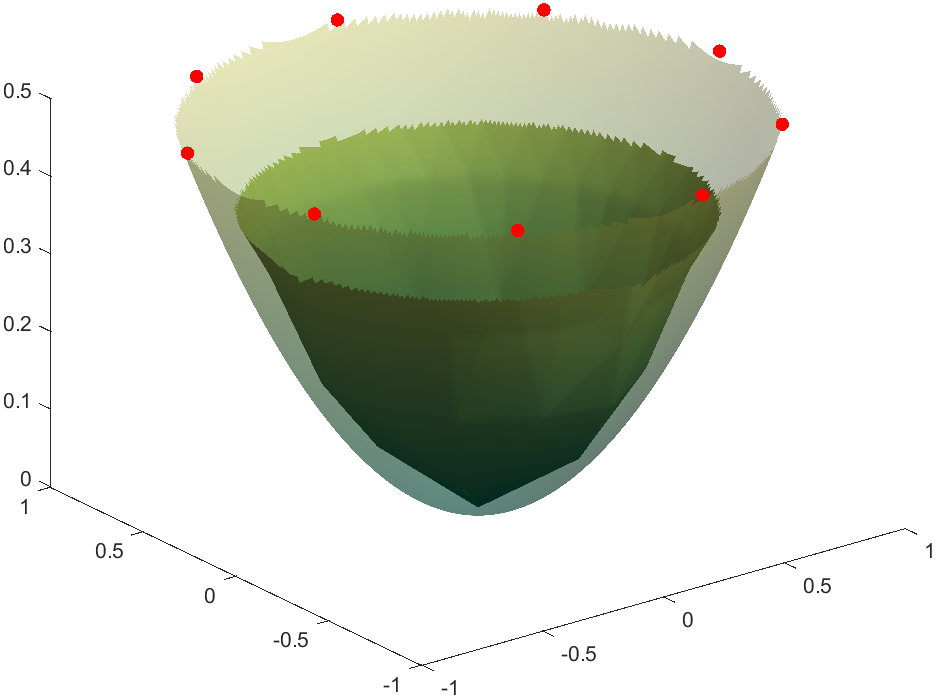}
    \caption{Solution of \eqref{eq:dirichlet-poisson-equation} on the graph in Figure \ref{fig:mold-forest}.}
    \label{fig:laplace-regularization-mold}
\end{figure}

\begin{table}[ht]
    \centering
    \begin{tabular}{c|c|c}
        $\max_{x\in \mathcal{X}\setminus \mathcal{O}}\left|u_{\text{scaled}}(x)-\frac{1}{2}|x|^2\right|$ & MA & Laplace \\
        \hline
        Random & 0.284112 & 0.277751\\
        Rays & 0.395477 & 0.240851 \\
        Spiral & 0.465382 & 0.475445 \\
        Uniform & 0.096614 & 0.183214 \\
        Radial tree & 0.035130 & 0.075822
    \end{tabular}
    
    \medskip
    
    \begin{tabular}{c|c|c}
        $\sqrt{\sum_{x\in\mathcal{X}\setminus \mathcal{O}}\left|u_{\text{scaled}}(x)-\frac{1}{2}|x|^2\right|^2}$ & MA & Laplace \\
        \hline
        Random & 0.802448 & 1.239437\\
        Rays & 0.869006 & 1.181113 \\
        Spiral & 2.133426 & 2.704896 \\
        Uniform & 0.336707 & 0.932423 \\
        Radial tree & 0.252256 & 0.599208
    \end{tabular}
    \medskip
    \caption{Errors with respect to the continuum solution.}
    \label{tab:error}
\end{table}

The graph-Laplacian solutions are generally flatter, with the largest
variations concentrated near the endpoints and central vertices.
Table~\ref{tab:error} shows that the graph Monge--Amp\`ere method has
a smaller discrete $\ell^2$ error in all five experiments. In the
maximum norm, it performs better for the spiral, uniform, and radial
graphs, whereas the graph Laplacian performs better for the random
and ray graphs. These experiments, therefore, suggest that the graph
Monge--Amp\`ere operator can provide more accurate qualitative
approximations on suitably organized tree-like graphs, particularly
for the uniform and radial constructions.

\subsection{Future work}

One possible direction for future work is to incorporate positive edge weights into the definition of the discrete eigenvalues. Let \(\omega_{xy}>0\) be a weight associated with the edge joining \(x\)
and \(y\). A weighted version of the eigenvalues could be defined by
\[
    \lambda_i^\omega[u](x)
    \coloneqq
    \min_{\substack{S\subseteq\mathcal N_x\\ |S|=2i}}
    \max_{\substack{y,z\in S\\ y\neq z}} \,
    \frac12 \left[
        \omega_{xy}
        \bigl(u(y)-u(x)\bigr)
        +
        \omega_{xz}
        \bigl(u(z)-u(x)\bigr)
    \right].
\]

Positive edge weights may be necessary to obtain a geometrically consistent scaling and a meaningful continuum limit. At present, the normalization is graph-dependent, so the numerical experiments assess qualitative shape rather than convergence.

\end{document}

%% file: TikZ/PeelingFigure.tikz
\tikzset{
    active/.style={
        circle, draw=blue!60!black, fill=blue!12,
        thick, minimum size=7mm, inner sep=0pt
    },
    peel/.style={
        circle, draw=orange!80!black, fill=orange!25,
        thick, minimum size=7mm, inner sep=0pt
    },
    removed/.style={
        circle, draw=gray!65, fill=gray!15,
        thick, minimum size=7mm, inner sep=0pt, text=gray!70
    },
    boundary/.style={
        circle, draw=red!70!black, fill=red!15,
        thick, minimum size=7mm, inner sep=0pt
    },
    oldedge/.style={line width=0.6pt, draw=gray!45},
    activeedge/.style={line width=1.1pt, draw=black},
    boundaryedge/.style={line width=0.6pt, draw=gray!55, dashed}
}

\begin{minipage}{0.47\textwidth}
\centering
\begin{tikzpicture}[scale=0.9]

\node at (1.8,2.05) {\small Step 0: find $A_0$};

\node[peel]   (a) at (0,0) {$a$};
\node[active] (b) at (1.2,0) {$b$};
\node[active] (c) at (2.4,0) {$c$};
\node[active] (d) at (3.6,0) {$d$};
\node[peel]   (e) at (2.4,1.1) {$e$};
\node[peel]   (f) at (1.2,-1.1) {$f$};
\node[peel]   (g) at (3.6,-1.1) {$g$};

\node[boundary] (o1) at (0,1.1) {$o_1$};
\node[boundary] (o2) at (2.4,-1.1) {$o_2$};

\draw[activeedge] (a)--(b);
\draw[activeedge] (b)--(c);
\draw[activeedge] (c)--(d);
\draw[activeedge] (c)--(e);
\draw[activeedge] (b)--(f);
\draw[activeedge] (d)--(g);

\draw[boundaryedge] (o1)--(a);
\draw[boundaryedge] (o1)--(e);
\draw[boundaryedge] (o2)--(f);
\draw[boundaryedge] (o2)--(g);

\node at (1.8,-1.85) {\scriptsize
$S_0=\mathcal{U},\quad A_0=\{a,e,f,g\}$};

\end{tikzpicture}
\end{minipage}
\hfill
\begin{minipage}{0.47\textwidth}
\centering
\begin{tikzpicture}[scale=0.9]

\node at (1.8,2.05) {\small Step 1: remove $A_0$};

\node[removed] (a) at (0,0) {$a$};
\node[peel]    (b) at (1.2,0) {$b$};
\node[active]  (c) at (2.4,0) {$c$};
\node[peel]    (d) at (3.6,0) {$d$};
\node[removed] (e) at (2.4,1.1) {$e$};
\node[removed] (f) at (1.2,-1.1) {$f$};
\node[removed] (g) at (3.6,-1.1) {$g$};

\node[boundary] (o1) at (0,1.1) {$o_1$};
\node[boundary] (o2) at (2.4,-1.1) {$o_2$};

\draw[oldedge] (a)--(b);
\draw[oldedge] (b)--(c);
\draw[oldedge] (c)--(d);
\draw[oldedge] (c)--(e);
\draw[oldedge] (b)--(f);
\draw[oldedge] (d)--(g);

\draw[boundaryedge] (o1)--(a);
\draw[boundaryedge] (o1)--(e);
\draw[boundaryedge] (o2)--(f);
\draw[boundaryedge] (o2)--(g);

\draw[activeedge] (b)--(c);
\draw[activeedge] (c)--(d);

\node at (1.8,-1.85) {\scriptsize
$S_1=\{b,c,d\},\quad A_1=\{b,d\}$};

\end{tikzpicture}
\end{minipage}

\medskip
\medskip

\begin{minipage}{0.47\textwidth}
\centering
\begin{tikzpicture}[scale=0.9]

\node at (1.8,2.05) {\small Step 2: remove $A_1$};

\node[removed] (a) at (0,0) {$a$};
\node[removed] (b) at (1.2,0) {$b$};
\node[peel]    (c) at (2.4,0) {$c$};
\node[removed] (d) at (3.6,0) {$d$};
\node[removed] (e) at (2.4,1.1) {$e$};
\node[removed] (f) at (1.2,-1.1) {$f$};
\node[removed] (g) at (3.6,-1.1) {$g$};

\node[boundary] (o1) at (0,1.1) {$o_1$};
\node[boundary] (o2) at (2.4,-1.1) {$o_2$};

\draw[oldedge] (a)--(b);
\draw[oldedge] (b)--(c);
\draw[oldedge] (c)--(d);
\draw[oldedge] (c)--(e);
\draw[oldedge] (b)--(f);
\draw[oldedge] (d)--(g);

\draw[boundaryedge] (o1)--(a);
\draw[boundaryedge] (o1)--(e);
\draw[boundaryedge] (o2)--(f);
\draw[boundaryedge] (o2)--(g);

\node at (1.8,-1.85) {\scriptsize
$S_2=\{c\},\quad A_2=\{c\}$};

\end{tikzpicture}
\end{minipage}
\hfill
\begin{minipage}{0.47\textwidth}
\centering
\begin{tikzpicture}[scale=0.9]

\node at (1.8,2.05) {\small Step 3: terminate};

\node[removed] (a) at (0,0) {$a$};
\node[removed] (b) at (1.2,0) {$b$};
\node[removed] (c) at (2.4,0) {$c$};
\node[removed] (d) at (3.6,0) {$d$};
\node[removed] (e) at (2.4,1.1) {$e$};
\node[removed] (f) at (1.2,-1.1) {$f$};
\node[removed] (g) at (3.6,-1.1) {$g$};

\node[boundary] (o1) at (0,1.1) {$o_1$};
\node[boundary] (o2) at (2.4,-1.1) {$o_2$};

\draw[oldedge] (a)--(b);
\draw[oldedge] (b)--(c);
\draw[oldedge] (c)--(d);
\draw[oldedge] (c)--(e);
\draw[oldedge] (b)--(f);
\draw[oldedge] (d)--(g);

\draw[boundaryedge] (o1)--(a);
\draw[boundaryedge] (o1)--(e);
\draw[boundaryedge] (o2)--(f);
\draw[boundaryedge] (o2)--(g);

\node at (1.8,-1.85) {\scriptsize
$S_3=\emptyset$};

\end{tikzpicture}
\end{minipage}

%% file: TikZ/SimpleEx.tikz
\begin{tikzpicture}[
    scale=1.3,
    vertex/.style={circle, draw, thick, minimum size=8mm, inner sep=0pt},
    unlabeled/.style={vertex, fill=blue!15},
    labeled/.style={vertex, fill=red!20},
    edge/.style={thick}
]

\node[unlabeled] (a) at (0,0) {$a$};
\node[unlabeled] (b) at (2,0) {$b$};
\node[labeled]   (c) at (1,1.7) {$c$};

\draw[edge] (a) -- (b);
\draw[edge] (b) -- (c);
\draw[edge] (c) -- (a);

\node at (0,-0.45) {$a\in \mathcal{U}$};
\node at (2,-0.45) {$b\in \mathcal{U}$};
\node at (1,2.15) {$c\in\mathcal O$};

\node at (1,-0.95) {$\mathcal{U}=\mathcal X\setminus\mathcal O=\{a,b\}$};

\end{tikzpicture}

%% file: TikZ/Forest.tikz
\begin{tikzpicture}[scale=0.65]
	\begin{pgfonlayer}{nodelayer}
		\node [style=Unlabeled] (0) at (-2.75, 0) {};
		\node [style=Unlabeled] (1) at (-2.75, 2) {};
		\node [style=Unlabeled] (2) at (-4, 2.25) {};
		\node [style=Unlabeled] (3) at (-1.75, 2.75) {};
		\node [style=Unlabeled] (4) at (-6, 3.75) {};
		\node [style=Unlabeled] (5) at (-4.5, 3.25) {};
		\node [style=Unlabeled] (6) at (-2.5, 3.75) {};
		\node [style=Unlabeled] (7) at (-0.5, 3.5) {};
		\node [style=Unlabeled] (8) at (-1.75, 4.25) {};
		\node [style=Unlabeled] (9) at (-3.75, 4.25) {};
		\node [style=Unlabeled] (10) at (3.75, 0) {};
		\node [style=Unlabeled] (11) at (4.25, 1) {};
		\node [style=Unlabeled] (12) at (2.75, 2.75) {};
		\node [style=Unlabeled] (14) at (1, 3) {};
		\node [style=Unlabeled] (15) at (1.75, 3.75) {};
		\node [style=Unlabeled] (16) at (5, 3.5) {};
		\node [style=Unlabeled] (17) at (3.75, 3.25) {};
		\node [style=Unlabeled] (18) at (5.75, 2.75) {};
		\node [style=Unlabeled] (19) at (2.25, 4.5) {};
		\node [style=Unlabeled] (20) at (3, 1.75) {};
		\node [style=Unlabeled] (21) at (0, 0) {};
		\node [style=Unlabeled] (22) at (-6.75, 0) {};
		\node [style=Unlabeled] (23) at (7, 0) {};
		\node [style=Labeled] (29) at (-5, 4.5) {};
		\node [style=Labeled] (30) at (-1, 2.5) {};
		\node [style=Labeled] (31) at (0.75, 4.25) {};
		\node [style=Labeled] (32) at (5, 2.25) {};
		\node [style=Labeled] (33) at (5, 5) {};
	\end{pgfonlayer}
	\begin{pgfonlayer}{edgelayer}
		\draw (2) to (1);
		\draw (1) to (0);
		\draw (1) to (3);
		\draw (3) to (6);
		\draw (3) to (7);
		\draw (6) to (8);
		\draw (5) to (9);
		\draw (5) to (4);
		\draw (5) to (2);
		\draw (11) to (10);
		\draw (16) to (18);
		\draw (15) to (19);
		\draw (15) to (14);
		\draw (15) to (12);
		\draw (20) to (11);
		\draw (20) to (12);
		\draw (17) to (20);
		\draw (17) to (16);
		\draw (21) to (0);
		\draw (0) to (22);
		\draw (10) to (23);
		\draw (10) to (21);
	\end{pgfonlayer}
\end{tikzpicture}

%% file: TikZ/WeavedForest.tikz
\begin{tikzpicture}[scale=0.65]
	\begin{pgfonlayer}{nodelayer}
		\node [style=Unlabeled] (0) at (-2.75, 0) {};
		\node [style=Unlabeled] (1) at (-2.75, 2) {};
		\node [style=Unlabeled] (2) at (-4, 2.25) {};
		\node [style=Unlabeled] (3) at (-1.75, 2.75) {};
		\node [style=Unlabeled] (4) at (-6, 3.75) {};
		\node [style=Unlabeled] (5) at (-4.5, 3.25) {};
		\node [style=Unlabeled] (6) at (-2.5, 3.75) {};
		\node [style=Unlabeled] (7) at (-0.5, 3.5) {};
		\node [style=Unlabeled] (8) at (-1.75, 4.25) {};
		\node [style=Unlabeled] (9) at (-3.75, 4.25) {};
		\node [style=Unlabeled] (10) at (3.75, 0) {};
		\node [style=Unlabeled] (11) at (4.25, 1) {};
		\node [style=Unlabeled] (12) at (2.75, 2.75) {};
		\node [style=Unlabeled] (14) at (1, 3) {};
		\node [style=Unlabeled] (15) at (1.75, 3.75) {};
		\node [style=Unlabeled] (16) at (5, 3.5) {};
		\node [style=Unlabeled] (17) at (3.75, 3.25) {};
		\node [style=Unlabeled] (18) at (5.75, 2.75) {};
		\node [style=Unlabeled] (19) at (2.25, 4.5) {};
		\node [style=Unlabeled] (20) at (3, 1.75) {};
		\node [style=Unlabeled] (21) at (0, 0) {};
		\node [style=Unlabeled] (22) at (-6.75, 0) {};
		\node [style=Unlabeled] (23) at (7, 0) {};
		\node [style=Labeled] (24) at (-5, 4.5) {};
		\node [style=Labeled] (25) at (5, 5) {};
		\node [style=Labeled] (26) at (0.75, 4.25) {};
		\node [style=Labeled] (27) at (5, 2.25) {};
		\node [style=Labeled] (28) at (-1, 2.5) {};
	\end{pgfonlayer}
	\begin{pgfonlayer}{edgelayer}
		\draw (2) to (1);
		\draw (1) to (0);
		\draw (1) to (3);
		\draw (3) to (6);
		\draw (3) to (7);
		\draw (6) to (8);
		\draw (5) to (9);
		\draw (5) to (4);
		\draw (5) to (2);
		\draw (11) to (10);
		\draw (16) to (18);
		\draw (15) to (19);
		\draw (15) to (14);
		\draw (15) to (12);
		\draw (20) to (11);
		\draw (20) to (12);
		\draw (17) to (20);
		\draw (17) to (16);
		\draw (21) to (0);
		\draw (0) to (22);
		\draw (10) to (23);
		\draw (10) to (21);
		\draw (28) to (21);
		\draw (28) to (7);
		\draw (28) to (14);
		\draw (26) to (14);
		\draw (26) to (19);
		\draw (26) to (7);
		\draw (27) to (11);
		\draw (27) to (18);
		\draw (27) to (16);
		\draw (27) to (17);
		\draw (25) to (18);
		\draw (25) to (19);
		\draw (24) to (4);
		\draw (24) to (9);
		\draw (28) to (8);
		\draw [bend left=45] (24) to (8);
		\draw [in=60, out=165, looseness=1.50] (24) to (22);
		\draw [bend right=60, looseness=2.25] (28) to (9);
		\draw [bend right=285] (28) to (4);
		\draw [bend left] (25) to (23);
		\draw (27) to (23);
		\draw [bend left=75, looseness=1.50] (28) to (6);
		\draw [bend left=45] (24) to (2);
		\draw [in=30, out=-105] (28) to (22);
		\draw [bend right] (12) to (27);
	\end{pgfonlayer}
\end{tikzpicture}